\RequirePackage{fix-cm}
\documentclass[smallextended]{svjour3}       
\smartqed  

\usepackage{abstract}
\usepackage{amsfonts}
\usepackage{amsmath}
\usepackage{amssymb}
\usepackage{amsrefs}
\usepackage[titletoc]{appendix}
\usepackage{appendix}
\usepackage{verbatim}
\usepackage{float}
\usepackage{graphicx}
\usepackage{color}
\usepackage{xcolor}

\newcommand{\eps}{\varepsilon}
\newcommand{\e}{\varepsilon}

\newtheorem{scheme}{Scheme}

\newcommand{\ZZH}[1]{\textcolor{blue}{#1}}

\usepackage{mathptmx}      
\begin{document}

\title{Fast Algorithms and Error Analysis of Caputo Derivatives with  Small Factional Orders
}


\author{Zihang Zhang  \and
        Qiwei Zhan       \and
        Zhennan Zhou
}


\institute{Zihang Zhang \at
           School of Mathematical Sciences, Peking University, Beijing 100871, China\\
           ORCID: 0000-0002-3630-0115\\
           \email{zhang-zihang@pku.edu.cn}           
           \and
           Qiwei Zhan \at
              School of Information Science and Electronic Engineering, Zhejiang University, Hangzhou, Zhejiang 310027, China\\
              ORCID: 0000-0001-7500-6157\\
              \email{qwzhan@zju.edu.cn (corresponding author)}           
           \and
           Zhennan Zhou \at
           Beijing International Center for Mathematical Research, Peking University, Beijing 100871, China\\
           ORCID: 0000-0003-4822-0275\\
           \email{zhennan@bicmr.pku.edu.cn (corresponding author)}           
}

\date{Received: date / Accepted: date}

\maketitle

\begin{abstract}
In this paper, we  investigate  fast algorithms {in the small fraction order regime} to approximate the Caputo derivative $^C_0D_t^\alpha u(t)$ when $\alpha$ is small. We focus on two fast algorithms, i.e. FIR and FIDR, both relying on the sum-of-exponential approximation to reduce the cost of evaluating the history part. FIR is the numerical scheme originally proposed in \cite{JiangZhang}, and FIDR is an alternative scheme {proposed in \cite{shen2018fast}}, and {we show that the latter is superior} when $\alpha$ is small. With quantitative estimates, we prove that given a certain error threshold, the computational cost of evaluating the history part of the Caputo derivative can be decreased as $\alpha$ gets small. Hence, only minimal cost for the fast evaluation is required in the small $\alpha$ regime, which matches prevailing protocols in engineering practice. We also present improved stability and error analysis of FIDR for solving linear fractional diffusion equations, {which achieves 
clear dependence of the error bound on the fraction order $\alpha$.} 
Finally, we carry out systematic numerical studies for the performances of both FIR and FIDR schemes, where we explore the trade-off between accuracy and efficiency when $\alpha$ is small.
\end{abstract}

\section{Introduction}

In recent years, there has been an emerging interest in the field of fractional derivatives. Many phenomena in engineering have been described by models with fractional derivatives, such as the groundwater flow \cite{Atangana, Atangana2}, the blood ethanol concentration system \cite{SAAMD}, the epidemic model \cite{ASMADA}, the magnetic hysteresis phenomena \cite{CaputoFabrizio}, and seismic wave propagation problems discussed in \cite{ZhanZhuangLiu}. {These complex media or processes in these fields need hereditary descriptions, while fractional derivative is an excellent mathematical tool for characterizing the memory effects \cite{SunChenWei, ZHAO2019531}.} 

Several versions of fractional time derivative has been proposed, including the Riemann-Liouville (RL) fractional derivative \cite[Sec.~2]{SKM}, \cite{SousaLi}, the Gr\"unwald-Letnikov (GL) fractional derivative \cite[Sec.~2.2]{Podlubny}, \cite{WangTreena}, and the
Caputo fractional derivative \cite{Caputo1}.  Both GL fractional derivative and RL fractional derivative require fractional-type initial values, whose physical interpretation is not clear. On the other hand, the Caputo fractional derivative takes the integer-order differential equations as the initial value. We refer to  \cite{Podlubny,KST,MKM} for a more general discussion.

The general form of the Caputo fractional derivative is represented as follows
\begin{equation}\label{caputo}
	^C_0D_t^\alpha u(t)=\frac{1}{\Gamma(m-\alpha)}\int_0^t \frac{ \partial^m_{\tau} u (\tau) }{(t-\tau)^{\alpha+1-m}}d\tau, \qquad m-1\leq\alpha<m,\ m\in \mathbb Z.
\end{equation}
In this paper, we focus on the numerical approximation of the Caputo derivative with $\alpha \in (0,1)$; then \eqref{caputo} becomes
\begin{equation}\label{caputo1}
	^C_0D_t^\alpha u(t)=\frac{1}{\Gamma(1-\alpha)}\int_0^t\frac{u'(\tau)}{(t-\tau)^{\alpha}}d\tau, \qquad 0<\alpha<1.
\end{equation}Designing an efficient numerical methods for fractional differential equations is non-trivial, since the fractional derivative at $t$ depends on the information of all $u(\tau)$ on $\tau\in (0,t)$  \cite{LiuMaZhou}. {For heuristic purposes, we introduce the L1 approximation  below to demonstrate the computational burden due to the non-local dependence in time of the Caputo derivative.}

The L1 approximation, based on piecewise linear interpolation of $u$, is a popular scheme of discretizing the Caputo fractional derivative \cite{GSZ,GS,GS2}. The cost of the one-time evaluation of the time derivative grows linearly as the time step increases. For a given time grid $\{t_n,\ 0=t_0<t_1<\cdots <t_N=t\},\ t_k=k\Delta t$, the L1 approximation constructs a finite difference scheme as \cite[eq.~3.1]{GSZ}
	\begin{equation}\label{L1}
		^C_0D_t^\alpha u^n=\frac{\Delta t^{-\alpha}}{\Gamma(2-\alpha)}\left[a_0^{(\alpha)}u^n-\sum_{l=1}^{n-1}(a_{n-l-1}^{(\alpha)}-a_{n-l}^{(\alpha)})u^l-a_{n-1}^{(\alpha)}u^0\right],
	\end{equation}
	where $u^k:= u(t_k)$ and 
	\[a_l^{(\alpha)}=(l+1)^{1-\alpha}-l^{1-\alpha},\qquad l\geq 0.\]
	As aforementioned, the fractional derivative at $t_k$  depends on all $u(\tau)$ on $\tau\in (0,t)$; as a result, the finite difference scheme \eqref{L1} requires $O(n)$ computational cost to compute $C_0D_t^\alpha u^n$, and $O(n^2)$ computational cost to solve the fractional differential equations (since we need to compute all $^C_0D_t^\alpha u^k,\ k=1,\cdots,n$). Therefore, it is expensive when $n$ is large.

{There are plenty of works that propose fast evaluations of the Caputo derivative, including the second order implicit schemes \cite{GHJCA}, the fast evaluations with the sum-of-exponential approximation \cite{ZengTB, JiangZhang}, as well as some efficient algorithms implemented in engineering works \cite{ZhanZhuangSun2017,ZhanZhuangLiu,ZZFL}, \cite[Chapter~2]{Carcione2001}. In this paper we focus on the fast schemes using the sum-of-exponential approximation, which is a widely-used method in speeding up the evaluation of the convolution integrals.  Such approximations have been used in efficient approximations for heat kernel  \cite{ JiangGreengardWang}, the evaluation of average probability of error  \cite{LoskotBeaulieu}, and the evaluation of the exponential integral function  \cite{AlkheirIbnkahla}.  
	
	The fast scheme in \cite{JiangZhang} is presented in the following to illustrate the application of the sum-of-exponential approximation in the Caputo derivative.}  
To distinguish the scheme that we shall discuss in this paper, 
we call the scheme in \cite{JiangZhang} the fast evaluation of the integral representation, abbreviated by FIR.  To calculate the Caputo derivative \eqref{caputo1} when $t=t_n$, FIR splits the convolution integral in \eqref{caputo1} into two parts --- a local part containing the integral from $t_{n-1}$ to $t_n$, and a history part containing the integral from $0$ to $t_{n-1}$:
\begin{equation}\label{devide}
	\begin{aligned}
		^C_0D_t^\alpha (t_n)&=\frac{1}{\Gamma(1-\alpha)}\int_0^{t_n}\frac{u'(s)}{(t_n-s)^{\alpha}}ds\\
		&=\frac{1}{\Gamma(1-\alpha)}\int_{t_{n-1}}^{t_n}\frac{u'(s)}{(t_n-s)^{\alpha}}ds+\frac{1}{\Gamma(1-\alpha)}\int_{0}^{t_{n-1}}\frac{u'(s)}{(t_n-s)^{\alpha}}ds\\
		&:=C_l(t_n)+C_h(t_n).
	\end{aligned}
\end{equation}
FIR applies the  standard L1 approximation for the local part, with the integration by parts, the history part is rewritten as,
\begin{equation}
	\begin{aligned}
		C_h(t_n)&=\frac{1}{\Gamma(1-\alpha)}\int_{0}^{t_{n-1}}\frac{u'(s)}{(t_n-s)^{\alpha}}ds\\
		&=\frac{1}{\Gamma(1-\alpha)}\left[\frac{u(t_{n-1})}{\Delta t_n^\alpha}-\frac{u(t_0)}{t_n^\alpha}-\alpha\int_0^{t_{n-1}}\frac{u(s)ds}{(t_n-s)^{1+\alpha}}\right].
	\end{aligned}
\end{equation}

The main challenge to calculate the Caputo derivative efficiently is  the evaluation of the time integral from 0 to $t_{n-1}$.  {FIR uses the sum-of-exponential approximation to approximate the integral in the history part, instead of a  finite difference scheme,  thus the computational cost is considerably reduced.}  To be more specific, \citeauthory[eqs.~2.1-2.4]{JiangZhang} approximates the convolution kernel $t^{-1-\alpha}$ via a sum-of-exponentials: 
\begin{equation}\label{sumprox1}
	\int_0^{t_{n-1}}\frac{u(s)ds}{(t_n-s)^{1+\alpha}}\approx \sum_{i=1}^{N_{\text{exp}}}w_i\int_0^{t_{n-1}}u(\tau)e^{-s_i(t_n-\tau)}d\tau,
\end{equation}
where $s_i$ and $w_i$ are the nodes and weights of sum-of-exponentials, respectively, and $N_{\text{exp}}$ denotes the total number of modes. Note that $N_{\text{exp}} (\ll n)$ is chosen for a given error threshold. For each time step, FIR only needs to update the $N_{\text{exp}}$ modes to assemble the approximation of the history part instead of gathering contributions from all past time steps $u(t_k),\ k=1,\cdots, n$. Thus FIR reduces the storage requirement from $O(n)$ to $O(N_{\text{exp}})$. Furthermore, the
overall computational cost is reduced from $O(n^2)$ to $O(nN_{\text{exp}})$ in  \eqref{L1}.

This work however, is devoted to investigating the numerical representation of the Caputo derivative in the small $\alpha$ regime. This has been overlooked by the community of numerical analysis. It is motivated by quantitative comparison in the  recent work for  efficient viscoelastic wave modeling problems \cite{ZZFL,Kjartansson}, where the authors use the parameter $Q$ instead of $\alpha$ in \eqref{caputo} to quantify the memory effect of the fractional derivative,with the following relationship between $Q$ and $\alpha$ \cite[e.q.~2]{ZZFL} 
	\begin{equation}\label{defQ}
		\alpha=\frac{1}{\pi}\arctan \frac{1}{Q}.
	\end{equation}
	Therefore, $\alpha$ is roughly inversely proportional to  $Q$. In \cite{ZZFL}, the authors consider  time-domain
	anisotropic anelastic attenuation modeling involves Caputo fractional
	time derivatives, where an error of the factor function is defined as 
	\begin{equation}
		R(\omega,Q):=\cos^2\left(\frac{\pi \alpha}{2}\right)\left(\frac{\omega}{\omega_r}\right)^{2\alpha}\cos(\pi\alpha),
	\end{equation}
	where $\omega$ and $\omega_r$ denote the frequency to be computed and the reference angular frequency, respectively, and $\alpha$ is the same $\alpha$ as in \eqref{defQ}. They find that that the smaller the Q value is, the larger the error of $R(\omega,Q)$ is \cite[Fig.~1]{ZZFL}, \cite[Fig.~2]{ZhanZhuangLiu}. Similar phenomena can be find in the experiments \cite{Kjartansson},  which always show that the numerical experiments works better when $Q$ increases (i.e., $\alpha$ decreases). In complex engineering systems, the engineers can only afford a few modes, similar to $N_{\text{exp}}$ in \eqref{sumprox1}, to account for the history effect. Their results show that an acceptable approximation could be achieved with only $1$$\sim$ $3$ modes when $\alpha$ is small. And in their works it often suffices to consider the cases when $Q$ is relative large, like $Q=10,\ 30$ in \cite{Carcione}, respectively corresponding to $\alpha\approx 0.03,\ 0.01$.
On the other hands, the existing numerical analysis results in the applied math community have not addressed the dependence issue on $\alpha$; as a result, relatively impractical upper bound of the computational cost 
 is provided, when $\alpha<1$ is small \cite{JiLiao,GB} or $\alpha>1$  \cite{Wang},  thus beyond the computation capacity for 3D  real-world large-scale transient applications.

In this paper, we explore the relationship between $\alpha$ and the global error of fast algorithms for the Caputo derivative. For the algorithm FIR, we show that the global error will reduce as $\alpha$ gets smaller. 
Nevertheless, the fast scheme FIR still requires considerably many modes, even  $\alpha$ is moderately small; thus it is not affordable in complex  engineering applications. 

{We shall show that another scheme originally presented in \cite{shen2018fast} requires much smaller numbers of modes to achieve a satisfactory approximation error, when compared with FIR in the small $\alpha$ regime. In this paper, we call this scheme the fast evaluation of the integral differential representation, abbreviated by FIDR.}

{FIDR proposes a different evaluation for the history part. In FIDR,  the convolution integral is also split into the local part and the history part, similar to  \eqref{devide}. However, they use the sum-of-exponential approximation directly to evaluate the history part
	\begin{equation}\label{history1}
		C_h(t_n)=\frac{1}{\Gamma(1-\alpha)}\int_{0}^{t_{n-1}}\frac{u'(s)}{(t_n-s)^{\alpha}}ds.
	\end{equation}
	where the integral is less singular near $t-t_{n-1}$. }

{	In other words, they approximate the convolution kernel $t^{-\alpha}$ via sum-of-exponentials
	\begin{equation}\label{sumprox2}
		\int_{0}^{t_{n-1}}\frac{u'(s)}{(t-s)^{\alpha}}ds\approx \sum_{i=1}^{N_A}\tilde{w}_i\int_0^{t_{n-1}} e^{-\tilde{s}_i(t_n-\tau)}\partial_\tau u(\tau)d\tau.
	\end{equation}
The detailed construction of such a scheme will be presented in Section \ref{Sec-Cons}.	
	}
	
	
	{However, the small fraction order issue is not addressed in \cite{shen2018fast}, nor does the numerical analysis in \cite{shen2018fast} apply to this scenario. In fact, Theorem 4.2 in \cite{shen2018fast} gives an error bound that tends to $O(1)$ where $\alpha\rightarrow 0$.	}
	
	{We carry out an improved error analysis for FIDR, and obtain a sharper global error estimate with explicit dependence on the fraction order $\alpha$. In particular, it shows the FIDR scheme requires less history modes to achieve a certain accuracy when $\alpha$ gets smaller.}
	 
	{Furthermore, we also compare the two scheme FIR and FIDR and prove that the error of FIDR is smaller than the error of FIR when $N_A=N_{\text{exp}}$, where $N_A$ denotes the total number of modes in FIDR. The detailed analysis of the difference between two schemes is provided in Section \ref{Sec-exp}.}


{The novelties delivered from this paper are summarized below.
	\begin{enumerate}
		\item   We prove that  the global error of the scheme FIR {and FIDR} reduces, when $\alpha$ becomes smaller.  This serves as the first verification of the engineering protocol that reliable numerical experiments can be implemented with reduced cost in the small $\alpha$ regime (or equivalently, in the large $Q$ value regime). 
		\item Both quantitative estimates and extensive numerical tests show that given a certain error threshold, FIDR can evaluate the Caputo derivative with little computational cost. 
		\item Compared with FIR, the global error of FIDR is smaller when the storage cost and the computational cost are the same, especially when $\Delta t$ is small.
	\end{enumerate}
}

The paper is organized as follows. In Section \ref{Sec-Cons}, we give the scheme construction and main results for the two schemes. In Section \ref{Sec-exp}, we analyse the sum-of-exponential approximation especially when $\alpha$ is small. In Section \ref{Sec-stab}, we give the proof of stability and convergence of the two scheme. We compare the two schemes and present the numerical results of the schemes solving fractional diffusion PDEs in Section \ref{Sec-num}. Finally we give the conclusion and remarks in Section \ref{Sec-conc}.


\section{FIR and FIDR Scheme construction}\label{Sec-Cons}

{In this work, we consider two types of  discretization of the Caputo derive \eqref{caputo} when $0<\alpha<1$, and we pay special attention to the asymptotic regime $0<\alpha\ll 1$.}

When $0<\alpha<1$, the Caputo derivative \eqref{caputo} becomes
\begin{equation}
	^C_0D_t^\alpha u(t)=\frac{1}{\Gamma(1-\alpha)}\int_0^t\frac{u'(\tau)}{(t-\tau)^{\alpha}}d\tau, \qquad 0<\alpha<1.
\end{equation}
Suppose that we want to evaluate the Caputo fractional derivative in  the interval $[0,T]$ over a set of time grids $\Omega_t:=\{t_n=n\Delta t,\ n=0,1,\cdots,N_T\}$ with {$T$ denoting the total simulation time $N_T$ denoting the total number of time steps and the time step} $\Delta t:=T/N_T$. 

{Note that, the Caputo derivative involves the time integration from the initial time to the current time. To avoid storing all the function value from $t_0$ to $t_{N_T}$, it is natural to split the convolution integral into a sum of a local part and a history
	part and compute the history part, with a reduced but accurate approximation, that is,}
\begin{equation}\label{caputo2}
	\begin{aligned}
		^C_0D_t^\alpha (t_n)&=\frac{1}{\Gamma(1-\alpha)}\int_0^{t_n}\frac{u'(s)}{(t-s)^{\alpha}}ds\\
		&=\frac{1}{\Gamma(1-\alpha)}\int_{t_{n-1}}^{t_n}\frac{u'(s)}{(t-s)^{\alpha}}ds+\frac{1}{\Gamma(1-\alpha)}\int_{0}^{t_{n-1}}\frac{u'(s)}{(t-s)^{\alpha}}ds\\
		&:=C_l(t_n)+C_h(t_n).
	\end{aligned}
\end{equation}
where the last equality defines the local and history parts, respectively. In both FIR and FIDR schemes, we apply the standard L1 approximation for the local part
\begin{equation}\label{local}
	C_l(t_n)\approx \frac{u(t_n)-u(t_{n-1})}{\Delta t_n \Gamma(1-\alpha)}\int_{t_{n-1}}^{t_n}\frac{1}{(t_n-s)^\alpha}ds=\frac{u(t_n)-u(t_{n-1})}{\Delta t_n^\alpha\Gamma(2-\alpha)}.
\end{equation}

However, the approximations of the history part are different in these schemes. In the scheme FIR presented in \cite{JiangZhang}, they apply the integration by parts to eliminate $u'(s)$ and have
\begin{equation}\label{history}
	\begin{aligned}
		C_h(t_n)&=\frac{1}{\Gamma(1-\alpha)}\int_{0}^{t_{n-1}}\frac{u'(s)}{(t-s)^{\alpha}}ds\\
		&=\frac{1}{\Gamma(1-\alpha)}\left[\frac{u(t_{n-1})}{\Delta t_n^\alpha}-\frac{u(t_0)}{t_n^\alpha}-\alpha\int_0^{t_{n-1}}\frac{u(s)ds}{(t_n-s)^{1+\alpha}}\right].
	\end{aligned}
\end{equation}
To approximate the history part, they approximate the convolution kernel $t^{-1-\alpha}$ 
via a sum-of-exponentials efficiently on the interval $[\delta ,T]$ ($\delta>0$) with the absolute error $\e$. That is, for all $\e>0$ there exist positive real numbers $s_i$ and $w_i$ $(i=1,\cdots,N_{\text{exp}})$ such that for $0<\alpha<1$,
\begin{equation}\label{history2}
	\begin{aligned}
		\left| \frac{1}{t^{1+\alpha}}-\sum_{i=1}^{N_{\text{exp}}}w_ie^{-s_it}\right| \leq \e,\quad 0<\delta\leq t\leq T.
	\end{aligned}
\end{equation}
Here, $s_i$ and $w_i$ are the nodes and weights derived from the hybrid use of the Gauss-Legendre quadrature and the Gauss-Jacobi quadrature. For details see Lemma \ref{exp2} and Lemma \ref{exp3}. 


Furthermore, the history part can be approximated by  $N_{\text{exp}}$ modes, denoted by $U_{\text{hist},i}$ respectively, whose time evolution can be effectively realized by its current value as well as  $u(t_k)$ from $k=n-1$ and $k=n-2$,
\begin{equation}\label{FIR}
	\begin{aligned}
		^{C}_0D_t^\alpha u(t_n)&=C_l(t_n)+\frac{1}{\Gamma(1-\alpha)}\left[\frac{u(t_{n-1})}{\Delta t_n^\alpha}-\frac{u(t_0)}{t_n^\alpha}-\alpha\int_0^{t_{n-1}}\frac{u(s)ds}{(t_n-s)^{1+\alpha}}\right]\\
		&\approx\frac{u(t_n)-u(t_{n-1})}{\Delta t_n^\alpha\Gamma(2-\alpha)} +\frac{1}{\Gamma(1-\alpha)}\left[\frac{u(t_{n-1})}{\Delta t_n^\alpha}-\frac{u(t_0)}{t_n^\alpha}-\alpha \sum_{i=1}^{N_{\text{exp}}}\omega_iU_{\text{hist},i}(t_n)\right]
	\end{aligned}
\end{equation}
where 
\begin{equation*}
	\begin{aligned}
		U_{\text{hist},i}(t_n)&=e^{-s_i\Delta t}U_{\text{hist},i}(t_{n-1})+\int_{t_{n-2}}^{t_{n-1}}e^{-s_i(t_n-\tau)}u(\tau)d\tau\\
		&\approx e^{-s_i\Delta t}U_{\text{hist},i}(t_{n-1})+ \frac{e^{-s_i\Delta t}}{s_i^2\Delta t}\left[(e^{-s_i\Delta t}-1+s_i\Delta t)u(t_{n-1})\right.\\
		&\qquad \left.+(1-e^{-s_i\Delta t}-e^{-s_i\Delta t}s_i\Delta t)u(t_{n-2})\right].
	\end{aligned}
\end{equation*}

Based on the effective mode representation of the history part, we are ready to present the complete scheme FIR in the following 
\begin{scheme}[Scheme FIR]\label{DFIR}
	{
		We define the discrete scheme
		\begin{equation}\label{DFIR1}
			\mathbb{D}_t^\alpha u^n:=\frac{u^n-u^{n-1}}{\Delta t_n^\alpha\Gamma(2-\alpha)} +\frac{1}{\Gamma(1-\alpha)}\left[\frac{u^{n-1}}{\Delta t_n^\alpha}-\frac{u^0}{t_n^\alpha}-\alpha \sum_{i=1}^{N_{\text{exp}}}\omega_i\mathbb{U}_{\text{hist},i}^n\right],
		\end{equation}
		where
		\begin{equation}
			\begin{aligned}
				\mathbb{U}_{\text{hist},i}^n&:=e^{-s_i\Delta t}\mathbb{U}_{\text{hist},i}^{n-1}+ \frac{e^{-s_i\Delta t}}{s_i^2\Delta t}\left[(e^{-s_i\Delta t}-1+s_i\Delta t)u^{n-1}\right.\\
				&\qquad \left.+(1-e^{-s_i\Delta t}-e^{-s_i\Delta t}s_i\Delta t)u^{n-2}\right].
			\end{aligned}
		\end{equation}
	}
\end{scheme}

{However, the shortage of Scheme \ref{DFIR} is that the error, from the sum-of-exponential approximation, rises rapidly when the time step $\Delta t\rightarrow 0$, and thus the global error becomes unsatisfactory for extremely small time steps. The details of this shortage will be elaborated in Section \ref{Sec-exp}
	while the corresponding numerical result in Section \ref{Sec-num}. 
	
{Next we consider an alternative scheme for fast evaluation of the Caputo derivative. This scheme was originally proposed in \cite{shen2018fast}, and in this paper, we call it FIDR (abbreviation for the fast evaluation of the integral differential representation). }

{In FIDR, we treat the history part in the original form,
	\begin{equation}\label{history3}
		C_h(t_n)=\frac{1}{\Gamma(1-\alpha)}\int_{0}^{t_{n-1}}\frac{u'(s)}{(t-s)^{\alpha}}ds.
\end{equation}}
And we approximate $t^{-\alpha}$ instead of $t^{-1-\alpha}$ 
via a sum-of-exponentials approximation efficiently on the interval $[\delta ,T]$, where the induced error is related to the approximation form by the following theorem.

\begin{theorem}\label{exp} 
	
	{Let $0<\delta<T$, and let $\e_0>0$ be the desired precision,  there exist $N_A=O((\log\frac{1}{\e}+\log\frac{T}{\delta})^2)$ which denotes the total number of modes, and positive real numbers $\tilde{s}_i$ and $\tilde{w}_i$ $(i=1,\cdots,N_{A})$ such that for $\quad 0<\delta\leq t\leq T$ and $\alpha>0$,
		\begin{equation}\label{theoremA1}
			\begin{aligned}
				\left| \frac{1}{t^{\alpha}}-\sum_{i=1}^{N_{A}}\tilde{w}_ie^{-\tilde{s}_it}\right| \leq \e_0.
			\end{aligned}
		\end{equation}
		Here $\tilde{s}_i$ and $\tilde{w}_i$ are the nodes and weights derived from the hybrid use of the Gauss-Legendre quadrature and the Gauss-Jacobi quadrature, for details see Lemma \ref{exp5} and Lemma \ref{exp6}.}
\end{theorem}

{In the discrete scheme of FIDR, $\delta\leq \Delta t$, the time step, $\e_0$ in this theorem is one part of the global error in theorem \eqref{FIDRerror}. }  We give the proof of Theorem \ref{exp} in Section \ref{Sec-exp} {and show that $\e_0$ reduces when $\alpha$ gets smaller in Corollary \ref{FIDR-alpha-error}}.

Combining \eqref{local}, \eqref{history3} and \eqref{theoremA1} together, we can derive an alternative mode representation of the history part. The time evolution of those modes can also be effectively realized by its current value as well as  $u(t_k)$ from $k=n-1$ and $k=n-2$:
\begin{equation}\label{DifScheme}
	\begin{aligned}
		^{C}_0D_t^\alpha u(t_n)&=\frac{1}{\Gamma(1-\alpha)}\int_{t_{n-1}}^{t_n} \frac{1}{(t_n-\tau)^\alpha}\partial_\tau u(\tau)d\tau+\frac{1}{\Gamma(1-\alpha)}\int_{0}^{t_{n-1}} \frac{1}{(t_n-\tau)^\alpha}\partial_\tau u(\tau)d\tau \\
		&\approx\frac{u(t_n)-u(t_{n-1})}{\Delta t_n^\alpha\Gamma(2-\alpha)}+ \frac{1}{\Gamma(1-\alpha)}\left(\sum_{i=1}^{N_{A}}\tilde{w}_i\tilde{\psi}(t_{n},\tilde{s}_i)\right),
	\end{aligned}
\end{equation}
where
\begin{equation}\label{discrete}
	\begin{aligned}
		\tilde{\psi}(t_{n},\tilde{s}_i)&=\int_0^{t_{n-1}} e^{-\tilde{s}_i(t_n-\tau)}\partial_\tau u(\tau)d\tau\\
		&=e^{-\tilde{s}_i\Delta t}\int_0^{t_{n-2}} e^{-\tilde{s}_i(t_{n-1}-\tau)}\partial_\tau u(\tau)d\tau+\int_{t_{n-2}}^{t_{n-1}} e^{-\tilde{s}_i(t_n-\tau)}\partial_\tau u(\tau)d\tau\\
		&=e^{-\tilde{s}_i\Delta t}\tilde{\psi}(t_{n-1},\tilde{s}_i)+\int_{t_{n-2}}^{t_{n-1}} e^{-\tilde{s}_i(t_n-\tau)}\partial_\tau u(\tau)d\tau\\
		&\approx e^{-\tilde{s}_i\Delta t}\tilde{\psi}(t_{n-1},\tilde{s}_i)+\frac{u(t_{n-1})-u(t_{n-2})}{\Delta t}\int_{t_{n-2}}^{t_{n-1}} e^{-\tilde{s}_i(t_n-\tau)}d\tau\\
		&=e^{-\tilde{s}_i\Delta t}\tilde{\psi}(t_{n-1},\tilde{s}_i)+\frac{(u(t_{n-1})-u(t_{n-2}))(1-e^{-\tilde{s}_i\Delta t})e^{-\tilde{s}_i\Delta t}}{\tilde{s}_i\Delta t}.
	\end{aligned}
\end{equation}

The complete scheme FIDR is given as follows. 
\begin{scheme}[Scheme  FIDR]\label{DFIDR}
	{We define the discrete scheme
		\begin{equation}\label{DFIDR1}
			\mathbb{D}_t^\alpha u^n:=\frac{u^n-u^{n-1}}{\Delta t_n^\alpha\Gamma(2-\alpha)}+ \frac{1}{\Gamma(1-\alpha)}\left(\sum_{i=1}^{N_{A}}\tilde{w}_i\tilde{\Psi}_i^n\right),
		\end{equation}
		where
		\begin{equation}
			\tilde{\Psi}_i^n:=e^{-\tilde{s}_i\Delta t}\tilde{\Psi}_i^{n-1}+\frac{(u^{n-1}-u^{n-2})(1-e^{-\tilde{s}_i\Delta t})e^{-\tilde{s}_i\Delta t}}{\tilde{s}_i\Delta t}.
	\end{equation}}
\end{scheme}

{Here is a heuristic explanation for performance differences of these two algorithms in the small $\alpha$ regime. In FIR, they do the integral by parts on the history part first and use sum-of-exponential approximation to evaluate the primal function $u(x)$ times $(-1-\alpha)$th-degree term $(t_n-s)^{-1-\alpha}$, see \eqref{sumprox1}. $(t_n-s)^{-1-\alpha}$ has advantage that it converges to 0 rapidly when $t_n$ (or $T$, as the same) goes to infinity, but it also blows up fast when $s\rightarrow t_n$.  In FIDR, we use the sum-of-exponential approximation directly to evaluate the differential function $u'(x)$ times $-\alpha$th-degree term $(t_n-s)^{-\alpha}$, see \eqref{sumprox2}. Since $0<\alpha\ll 1$, $(t_n-s)^{-\alpha}$ rises slowly when $s\rightarrow t_n$, and the sum-of-exponential approximation works better compared with FIR.}


{We remark that Scheme \ref{DFIDR} is a full numerical scheme of the Caputo derivative, which is compatible with the generic initial value problem, or initial boundary value problem. For the rest of the paper, we take the reaction diffusion equation as the study subject, and the extension to other systems are natural although there might be additional challenges for specific models.}

We consider the initial value problem of the linear fractional diffusion equation, see \cite{GSZ} and \cite{JiangZhang}. {Denote $\Omega=(x_l,x_r)$ and introduce the nonreflecting boundary condition derived
	in \cite{GSZ}, we have}
\begin{equation}\label{problem}
	\begin{aligned}
		& ^C_0D_t^\alpha u(x,t)=u_{xx}(x,t)+f(x,t), &x\in\Omega,\ t>0,\\
		& u(x,0)=x_0(x),&x\in \Omega,\\
		& \frac{\partial u(x,t)}{\partial x}=\frac{1}{\Gamma (1-\frac \alpha 2)}\int_0^t \frac{u_s(x,s)}{(t-s)^{\frac \alpha 2}}ds:=\ ^C_0D_t^{\frac \alpha 2} u(x,t), &x=x_l\\
		& \frac{\partial u(x,t)}{\partial x}=-\frac{1}{\Gamma (1-\frac \alpha 2)}\int_0^t \frac{u_s(x,s)}{(t-s)^{\frac \alpha 2}}ds:=- ^C_0D_t^{\frac \alpha 2} u(x,t), &x=x_r.\\
	\end{aligned}
\end{equation}
{According to \cite{GSZ}, the finite difference scheme for the problem \eqref{problem} be written in the following form. For two given positive integers $N_T$ and $N_S$, let $\{t_n\}_{n=0}^{N_T}$ be a equidistant
	partition of $[0,T]$ with $t_n=n\Delta t$ and $\Delta t= T/N_T$, and let $\{x_j\}_{j=0}^{N_S}$ be a partition of $(x_l,x_r)$ with $x_i=x_l+ih$ and $h=(x_r-x_l)/N_S$. Denote $u_i^n=u(x_i,t_n),\ f_i^n=f(x_i,t_n)$, and}
\begin{equation*}
	\begin{aligned}
		&\delta_x u_{i+\frac 12}^n=\frac{u_{i+1}^n-u_i^n}{h}\\
		&\delta_x^2 u_i^n=\frac{\delta_x u_{i+\frac 12}^n-\delta_x u_{i-\frac 12}^n}{h}.
	\end{aligned}
\end{equation*}
Then we have
\begin{equation}\label{dif1}
	\begin{aligned}
		& \mathbb{D}_t^\alpha u_i^n=\delta^2_x u_i^n+f_i^n, &1\leq i\leq N_S-1,1\leq n\leq N_T,\\
		&\mathbb{D}_t^\alpha u_0^n=\frac 2h\left[\delta_x u_{\frac 12}^n- \mathbb{D}_t^{\frac \alpha 2} u_0^n\right]+f_0^n,\\
		&\mathbb{D}_t^\alpha u_{N_S}^n=\frac 2 h \left[-\delta_x u_{N_S-\frac 12}^n-  \mathbb{D}_t^{\frac \alpha 2} u_{N_S}^n\right]+f_{N_S}^n,\\
		&u_i^0=u_0(x_i),&0\leq i\leq N_S.
	\end{aligned}
\end{equation}
For FIR and FIDR, $\mathbb{D}_t^\alpha u$ in \eqref{dif1} is the discrete scheme defined in \eqref{DFIR1} and \eqref{DFIDR1}, respectively.

The thorough numerical analysis for this problem will be carried out in Section \ref{Sec-stab}, where we will give the error estimates of for these two schemes. We list the main results here.

\begin{theorem}\label{FIRerror}
	{ Suppose $u(x,t)\in C_{x,t}^{4,2}([x_l,x_r]\times [0,T]) $ is the solution of problem \eqref{problem}. For two given positive integers $N_T$ and $N_S$, let $\{t_k\}_{k=0}^{N_T}$ be a equidistant partition of $[0,T]$ with $t_k=k \Delta t$ and $\Delta t :=T/N_T$. Let $\{x_i\}_{j=0}^{N_S}$ be a equidistant partition of $[x_l,x_r]$ with $x_i=x_l+ih$ and $h=(x_r-x_l)/N_S$. Let  $\{u_i^k|0\leq i\leq N_S,\ 0\leq k\leq N_T\}$  be the numerical solutions of problem \eqref{problem} obtained by the difference scheme \eqref{dif1} and the scheme FIR \eqref{FIR}. If we denote the global error by $e_i^k=u_i^k-u(x_i,t_k)$, then there exists a positive constant $c_2$ such that 
		\begin{equation}\label{Theorem2}
			\e_{\text{global}}:=\sqrt{\Delta t \sum_{k=1}^n||e^k||_{\infty}^2}\leq c_2(\Delta t^{2-\alpha}+h^2+{\alpha}\e),\qquad 1\leq n\leq N_T,
		\end{equation}
		Here, $\e$ is the error of the sum-of-exponential approximation in \eqref{history2}.
	}
\end{theorem}
We remark that the global error analysis of FIR was presented in \cite{JiangZhang}. However, \cite{JiangZhang} did not discuss the effect of $\alpha$ on the  global error. We show that when $\alpha$ is small, we can soften the restrict on $\e$ on the right hand of \eqref{Theorem2}.

\begin{theorem}\label{FIDRerror}
	{ Suppose $u(x,t)\in C_{x,t}^{4,2}([x_l,x_r]\times [0,T]) $ is the solution of problem \eqref{problem}. For two given positive integers $N_T$ and $N_S$, let $\{t_k\}_{k=0}^{N_T}$ be a equidistant partition of $[0,T]$ with $t_k=k \Delta t$ and $\Delta t :=T/N_T$. Let $\{x_i\}_{j=0}^{N_S}$ be a equidistant partition of $[x_l,x_r]$ with $x_i=x_l+ih$ and $h=(x_r-x_l)/N_S$.  Let $\{u_i^k|0\leq i\leq N_S,\ 0\leq k\leq N_T\}$  be the numerical solutions of problem \eqref{problem} obtained by the difference scheme \eqref{dif1} and the scheme FIDR \eqref{DifScheme}. If we denote the global error by $e_i^k=u_i^k-u(x_i,t_k)$, then there
		exists a positive constant $\tilde{c}_2$ such that
		\begin{equation}\label{eFIDR}
			\begin{aligned}
				\e_{\text{global}}:=\sqrt{\Delta t \sum_{k=1}^n||e^k||_{\infty}^2}\leq\tilde{c}_2(\Delta t^{2-\alpha}+h^2+\e_0),
			\end{aligned}
		\end{equation}
		Here, $\e_0$ is the error of the sum-of-exponential approximation in \eqref{theoremA1}.
	}
\end{theorem}
{We emphasize that the result in Theorem \ref{FIDRerror} is better than Theorem 4.2 in  \cite{shen2018fast}. when $\alpha$ is small. The error bound in  \cite{shen2018fast} tends to $O(1)$ when $\alpha\rightarrow 0$. In \eqref{eFIDR} above, when $\alpha\rightarrow 0$, the error bound tends to $O(\Delta t^2 +h^2 +\e_0)$, and it is further shown in Corollary \ref{FIDR-alpha-error} of Section \ref{Sec-exp} that $\e_0$ in \eqref{eFIDR} reduces as $\alpha$ gets smaller. }

Although the error bound for the scheme FIDR does not explicitly show its dependence on $\alpha$, we show by analyzing the sum-of-exponential approximation and extensive numerical experiments the the scheme FIDR in fact leads to improved accuracy in the small $\alpha$ regime. Theorem \ref{FIRerror} and \ref{FIDRerror} are proved in Section \ref{Sec-stab}, with detailed characterization of the two constants $c_2$ and  $\tilde{c}_2$.


\section{Sum-of-exponential Approximation}\label{Sec-exp}


{It is worth noting that the main difference between the scheme FIDR and FIR is they apply the sum-of-exponential approximations to different forms.  In this section we give the proofs of sum-of-exponential estimates \eqref{history2} and \eqref{theoremA1}. } We also validate such properties with systematic numerical tests in Section 5.
\subsection{Sum-of-exponentials  in FIR}
{In this part, we prove the estimate \eqref{history2} and explain why this sum-of-exponential does not work well when time step $\Delta t\rightarrow 0$. Proofs of the lemmas in this subsection were already given in \cite{JiangZhang}, and the results are listed below for the comparison with the counterparts in FIDR.} 

We sketch the proof of the estimate \eqref{history2} as follows. First we transform $1/t^\beta$ into an integral form by Lemma \ref{exp0}. Then we split the integral interval into $[0,2^{-m}]$, $[2^{-m}, 2^{-m+1}], \cdots, [2^{n-1},2^n]$,  $[2^n,\infty]$. Lemma \ref{exp1} shows that integral on $[2^n,\infty]$ can be ignored when $n$ is large enough. And finally, integrals on $[0,2^{-m}]$, $[2^{i},2^{i+1}]$ ($i=-m,\cdots,n-1$) can be approximated by sum-of-exponentials based on Lemma \ref{exp3} and \ref{exp2}, respectively. 

We start with the following integral representation of the power function.

\begin{lemma}\label{exp0}
	For any $\beta>0$, $t>0$, 
	\begin{equation}\label{eq11} 
		\frac{1}{t^{\beta}}=\frac{1}{\Gamma(\beta)}\int_0^{\infty}e^{-ts}s^{\beta-1}ds.
	\end{equation}
\end{lemma} 
Note that \eqref{eq11} can be viewed as a representation of $t^{-\beta}$ using an infinitely many (continuous) exponentials. In order to obtain an efficient sum-of-exponentials approximation,
we first truncate the integral to a finite interval, then subdivide the finite interval
into a set of dyadic intervals and discretize the integral on each dyadic interval with
proper quadratures.\\ \\
We now assume $1<\beta<2$, which is the case we are concerned with in \eqref{history2}. 

\begin{lemma}\label{exp1}
	For $0<\delta\leq t$, $1<\beta<2$,
	\begin{equation}\label{error1}
		\left| \frac{1}{\Gamma(\beta)}\int_p^{\infty}e^{-ts}s^{\beta-1}ds \right|\leq e^{-\delta p}2^{\beta -1}\left(\frac{p^{\beta}}{\Gamma(\beta)}+\frac{1}{\delta^\beta}\right).
	\end{equation}
\end{lemma} 
{ Lemma \ref{exp1} shows that when $p$ is fixed, the smaller $t$ is, the larger the truncation error will be, which may lead to an accuracy issue in numerical experimentation. In this lemma, a prescribed lower bound $\delta$ of $t$ gives the upper bound of truncation error. However, in Scheme \ref{DFIR}, the lower bound $\delta$ is the time step $\Delta t$. Thus the truncation error will become larger when the time step decreases. To illustrate the error, we present the number values of the left term of \eqref{error1} when $t$ is small in Table \ref{table1} below, which shows how the truncation error blows up when $t$ decreases. (0 in Table ~\ref{table1}  means the numerical value is less than $10^{-15}$.) We shall see that the global error blows up at the similar time steps in Figure \ref{Fig-lin01}, \ref{Fig-lin05} and \ref{Fig-lin07} of Section \ref{Sec-num},}
\begin{table}[H] 
	\centering
	\begin{tabular}{c|cccc}
		\hline
		& $p=2^5$ & $p=2^{10}$ &$p=2^{15}$& $p=2^{20}$ \\
		\hline
		$t=2^{-5}$ &1.859e$+$01 & 8.546e$-$13&0&0 \\
		\hline
		$t=2^{-6}$ &6.339e$+$01&1.523e$-$05&0&0 \\
		\hline
		$t=2^{-7}$&1.699e$+$02&9.129e$-$02 &0&0 \\
		\hline
		$t=2^{-8}$&4.052e$+$02 &1.006e$+$01&0&0 \\
		\hline
		$t=2^{-9}$ &9.136e$+$02&1.511e$+$02 &0&0 \\
		\hline
		$t=2^{-10}$ &2.005e$+$03&8.414e$+$02&3.867e$-$11&0 \\
		\hline
	\end{tabular}
	\caption{The numerical values of the left term of \eqref{error1} with different $t$ and $p$, {here $\beta=1.1$.}}  
	\label{table1}
\end{table}

The next two lemmas show how to choose the weights and nodes. But we need to be careful that the weights in \eqref{GL} and \eqref{GJ} are not the weights we need in \eqref{history2}.
\begin{lemma}\label{exp2} 
	Consider a dyadic interval $[a,b] = [2^j,2^{j+1}]$ and let $s_1,\cdots,s_n$ and $\omega_1, \cdots, \omega_n$ be the nodes and weights for n-point Gauss-Legendre quadrature on the interval. Then for $\beta\in (1,2),\ t>0$ and $n>1$,
	\begin{equation}\label{GL}
		\left| \int_a^{b}e^{-ts}s^{\beta-1}ds-\sum_{k=1}^{n}\omega_k s_k^{\beta-1}e^{-s_k t} \right|\leq 2^{\beta -\frac{3}{2}}\pi a^{\beta}\left(\frac{e^{1/e}}{4}\right)^{2n}.
	\end{equation} 
\end{lemma}

\begin{lemma}\label{exp3} 
	Let $s_1,\cdots,s_n$ and $\omega_1, \cdots, \omega_n$ ($n\geq 2$) be the nodes and weights for
	n-point Gauss-Jacobi quadrature with the weight function $s^{\beta-1}$ on the interval. Then for $0<t<T$, $\beta \in (1,2)$ and $n>1$,
	\begin{equation}\label{GJ}
		\left| \int_0^{a}e^{-ts}s^{\beta-1}ds-\sum_{k=1}^{n}\omega_k e^{-s_k t} \right|<2\sqrt{\pi} a^{\beta} n^{3/2}\left(\frac{e}{8}\right)^{2n}\left(\frac{aT}{n}\right)^{2n}.
	\end{equation}
\end{lemma}

{
	Finally, combining Lemma \ref{exp0}, \ref{exp1}, \ref{exp2} and \ref{exp3}, we can get \eqref{history2}.  An upper bound of $\e$ in \eqref{history2}  is
	\begin{multline}\label{serrorFIR}
		\e\leq e^{-\delta 2^n}2^{\beta -1} \left(\frac{2^{\beta n}}{\Gamma(\beta)}+\frac{1}{\delta^\beta} \right)
	\\	+ \frac{1}{\Gamma(\beta)}\left[2\sqrt{\pi} 2^{-\beta m} n_1^{3/2}\left(\frac{e}{8}\right)^{2n_1}\left(\frac{2^{-m}T}{n_1}\right)^{2n_1}+2^{\beta -\frac{3}{2}}\pi 2^{\beta n}\left(\frac{e^{1/e}}{4}\right)^{2n_2}\right],
	\end{multline}
	where $n_1,\ n_2$ is the number of nodes in lemma \ref{exp3}, \ref{exp2}, respectively. }

\subsection{Sum-of-exponentials in FIDR}
Here we give the proof of Theorem \ref{exp}. For convinience, we rewrite the theorem here: 

{Let $0<\delta<T$, and let $\e_0>0$ be the desired precision,  there exist $N_A=O((\log\frac{1}{\e}+\log\frac{T}{\delta})^2)$ which denotes the total number of modes, and positive real numbers $\tilde{s}_i$ and $\tilde{w}_i$ $(i=1,\cdots,N_{A})$ such that for $\quad 0<\delta\leq t\leq T$ and $\alpha>0$,
	\begin{equation}
		\begin{aligned}
			\left| \frac{1}{t^{\alpha}}-\sum_{i=1}^{N_{A}}\tilde{w}_ie^{-\tilde{s}_it}\right| \leq \e_0.
		\end{aligned}
	\end{equation}
	Here $\tilde{s}_i$ and $\tilde{w}_i$ are the nodes and weights derived from the hybrid use of the Gauss-Legendre quadrature and the Gauss-Jacobi quadrature, for details see Lemma \ref{exp5} and Lemma \ref{exp6}.}

Recall that by Lemma \ref{exp0}, we have
\begin{equation}\label{lemmaC1}
	\frac{1}{t^\alpha}=\frac{1}{\Gamma(\alpha)}\int_0^\infty e^{-ts}s^{\alpha-1}ds.
\end{equation}

\begin{lemma}\label{exp4} 
	For any $t\geq\delta>0$, $\alpha>0$,
	\begin{equation}\label{lemmaC2}
		\left|\frac{1}{\Gamma(\alpha)}\int_{p}^\infty e^{-ts}s^{\alpha-1}ds\right|\leq\frac{e^{-\delta p}}{\Gamma(\alpha)\delta p^{1-\alpha}}
	\end{equation}
\end{lemma}

\begin{proof}
	By direct calculations, we have
	\begin{equation*}
		\begin{aligned}
			\left|\frac{1}{\Gamma(\alpha)}\int_{p}^\infty e^{-ts}s^{\alpha-1}ds\right|&=\left|\frac{e^{-tp}}{\Gamma(\alpha)}\int_{0}^\infty e^{-ts}(s+p)^{\alpha-1}ds\right|\\
			&\leq\left|\frac{e^{-tp}}{\Gamma(\alpha)}\int_{0}^\infty e^{-ts}p^{\alpha-1}ds\right|\\
			&=\left|\frac{e^{-tp}p^{\alpha-1}}{\Gamma(\alpha)t}\right|\leq\frac{e^{-\delta p}}{\Gamma(\alpha)\delta p^{1-\alpha}},
		\end{aligned}
	\end{equation*}
\end{proof}
Compare Lemma \ref{exp4} and Lemma \ref{exp1}, we conclude that when $\delta$ is small, $\e_0=O(\frac 1\delta)$ in FIDR while $\e=O(\frac {1}{\delta^\beta})$. Cause $\beta>1$, sum-of-exponentials in FIDR will get better result if $\delta$ is extremely small. {As we mentioned in theorem \ref{exp}, $\delta\leq \Delta t$.  Thus the scheme of FIDR works better when the time step $\Delta t$ is small, which is shown in the numerical experiments in Section 5. }

\begin{lemma}\label{exp5}
	Consider a dyadic interval $[a,b]=[2^j,2^{j+1}]$ and let $s_1,\cdots, s_n$ and $w_1,\cdots,w_n$ be the nodes and weights for n-point Gauss-Legendre quadrature on the interval. Then for $\alpha\in(0,1)$ and $n>1$,
	\begin{equation}
		\left|\int_{a}^be^{-ts}s^{\alpha-1}ds-\sum_{k=1}^nw_ks_k^{\alpha-1}e^{-s_kt}\right|<2\sqrt{2}\pi a^{\alpha}\left(\frac{e^{1/e}}{4}\right)^{2n}.
	\end{equation}
\end{lemma}

\begin{proof}
	Based on formula (3.5.27) in \cite{OLBC}, the standard estimate for n-point Gauss-Legendre
	quadrature yields,
	\begin{equation}\label{approxGL}
		\left|\int_{a}^be^{-ts}s^{\alpha-1}ds-\sum_{k=1}^nw_ks_k^{\alpha-1}e^{-s_kt}\right|=\frac{(b-a)^{2n+1}}{2n+1}\frac{(n!)^4}{[(2n)!]^3}\left| g^{(2n)}(s)\right|, \quad s\in (a,b),
	\end{equation}
	where $g(s)=e^{-st}s^{\alpha-1}$.\\
	Applying Stirling’s approximation
	\begin{equation}
		\sqrt{2\pi}n^{n+1/2}e^{-n}<n!<2\sqrt{\pi}n^{n+1/2}e^{-n}
	\end{equation}
	
	\begin{equation}
		\begin{aligned}
			\left| g^{(2n)}(s)\right|&=\left|\sum_{k=0}^{2n}\left(\begin{matrix}
				2n\\k
			\end{matrix}\right)(D_s^{2n-k}e^{-s t})(D_s^ks^{\alpha-1})\right|\\
			&=\left|\sum_{k=0}^{2n}\left(\begin{matrix}
				2n\\k
			\end{matrix}\right)(-t)^{2n-k}e^{-st}P_{\alpha-1}^k s^{\alpha-k-1}\right|\\
			&=\left|\sum_{k=0}^{2n}\left(\begin{matrix}
				2n\\k
			\end{matrix}\right)P_{k-\alpha}^kt^{2n-k}e^{-st} s^{\alpha-k-1}\right|,\\
		\end{aligned}
	\end{equation}
	and
	\begin{equation}
		\begin{aligned}
			\left|\sum_{k=0}^{2n}\left(\begin{matrix}
				2n\\k
			\end{matrix}\right)P_{k-\alpha}^kt^{2n-k}e^{-st} s^{\alpha-k-1}\right|&\leq\left|e^{-st} s^{\alpha-1}\sum_{k=0}^{2n}\left(\begin{matrix}
				2n\\k
			\end{matrix}\right)k!t^{2n-k}s^{-k}\right|\\
			&\leq\left|e^{-st} s^{\alpha-1}\sum_{k=0}^{2n}\left(\begin{matrix}
				2n\\k
			\end{matrix}\right)(2\sqrt{\pi}k^{k+1/2}e^{-k})t^{2n-k}s^{-k}\right|\\
			&\leq\left|e^{-st} s^{\alpha-1}\sum_{k=0}^{2n}\left(\begin{matrix}
				2n\\k
			\end{matrix}\right)(2\sqrt{\pi}(2n)^{k+1/2}e^{-k})t^{2n-k}s^{-k}\right|\\
			&=\left|2\sqrt{2n\pi}e^{-st} s^{\alpha-1}\left(t+\frac{2n}{es}\right)^{2n}\right|.
		\end{aligned}
	\end{equation}
	Thus we have
	\begin{equation}\label{approxf}
		\left| g^{(2n)}(s)\right|\leq \left|2\sqrt{2n\pi}e^{-st} s^{\alpha-1}\left(t+\frac{2n}{es}\right)^{2n}\right|.
	\end{equation}
	Meanwhile, based on Stirling’s approximation,
	\begin{equation}\label{approxn}
		\begin{aligned}
			\frac{(n!)^4}{[(2n)!]^3}<2\sqrt{\pi}\left(\frac{e}{8}\right)^{2n}\frac{\sqrt{n}}{n^{2n}}.
		\end{aligned}
	\end{equation}
	Taking \eqref{approxf}, \eqref{approxn} into \eqref{approxGL}, and recall that $b=2a$, we have
	\begin{equation}\label{approxGL2}
		\begin{aligned}
			\left|\int_{a}^be^{-ts}s^{\alpha-1}ds-\sum_{k=1}^nw_ks_k^{\alpha-1}e^{-s_kt}\right|&\leq\frac{(b-a)^{2n+1}}{2n+1}\frac{(n!)^4}{[(2n)!]^3}\max_{a<s<b}\left| g^{(2n)}(s)\right|\\
			&<\frac{a^{2n+1}4\sqrt{2}\pi n}{2n+1}e^{-at}a^{\alpha-1}\left(\frac{et}{8n}+\frac{1}{4a}\right)^{2n}\\
			&=\frac{4\sqrt{2}\pi n}{2n+1}a^{\alpha}e^{-at}\left(\frac{eat}{8n}+\frac{1}{4}\right)^{2n}.
		\end{aligned}
	\end{equation}
	And we have 
	\begin{equation*}
		\max_{x>0}e^{-x}\left(\frac{ex}{8n}+\frac{1}{4}\right)^{2n}=\left(\frac{e^{1/e}}{4}\right)^{2n},\qquad n\geq 2,
	\end{equation*}
	\eqref{approxGL2} becomes,
	\begin{equation}
		\begin{aligned}
			\left|\int_{a}^be^{-ts}s^{\alpha-1}ds-\sum_{k=1}^nw_ks_k^{\alpha-1}e^{-s_kt}\right|< 2\sqrt{2}\pi a^{\alpha}\left(\frac{e^{1/e}}{4}\right)^{2n}.
		\end{aligned}
	\end{equation}
\end{proof} 

\begin{lemma}\label{exp6}
	let $s_1,\cdots, s_n$ and $w_1,\cdots,w_n$ be the nodes and weights for n-point Gauss-Jacobi quadrature with the weight function $s^{\alpha-1}$ on the interval. Then for $0<t<T,\ \alpha\in(0,1)$ and $n>1$,
	\begin{equation}
		\begin{aligned}
			\left|\int_{0}^ae^{-ts}s^{\alpha-1}ds-\sum_{k=1}^nw_ke^{-s_kt}\right|<\frac{4\sqrt{\pi}a^{\alpha}}{e^2}\frac{(2n-1)n^{3/2}}{(2n+\alpha)}\left[\frac{aenT}{2(2n-1)^2}\right]^{2n}.
		\end{aligned}
	\end{equation}
\end{lemma}

\begin{proof}
	Based on formula 3.5.26 in \cite{OLBC}, the standard estimate for n-point Gauss-Jacobi
	quadrature yields,
	\begin{equation}
		\begin{aligned}
			\left|\int_{0}^ae^{-ts}s^{\alpha-1}ds-\sum_{k=1}^nw_ke^{-s_kt}\right|=\frac{a^{2n+\alpha}}{2n+\alpha}\frac{(n!)^2[\Gamma(n+\alpha)]^2}{(2n!)[\Gamma(2n+\alpha)]^2}\left|D_s^{2n}e^{-st}\right|,\qquad s\in(0,a)
		\end{aligned}
	\end{equation}
	where we have $\Gamma(n+\alpha)<\Gamma(n+1)=n!,\ \Gamma(2n+\alpha)>\Gamma(2n)=(2n-1)!$. Thus
	\begin{equation*}
		\begin{aligned}
			\left|\int_{0}^ae^{-ts}s^{\alpha-1}ds-\sum_{k=1}^nw_ke^{-s_kt}\right|&<\frac{a^{2n+\alpha}}{2n+\alpha}\frac{(n!)^2(n!)^2}{(2n!)[(2n-1)!]^2}t^{2n}e^{-st}\\
			&\leq\frac{a^{2n+\alpha}}{2n+\alpha}\frac{4\sqrt{\pi}(2n-1)n^{3/2}}{e^2}\left[\frac{en}{2(2n-1)^2}\right]^{2n}t^{2n}e^{-st}\\
			&\leq\frac{4\sqrt{\pi}a^{\alpha}}{e^2}\frac{(2n-1)n^{3/2}}{(2n+\alpha)}\left[\frac{aenT}{2(2n-1)^2}\right]^{2n}.
		\end{aligned}
	\end{equation*}
\end{proof}
Finally,  we collect the results of Lemma \ref{exp4}, \ref{exp5} and \ref{exp6} and get Theorem \ref{exp}. 

We also give a upper bound of $\e_0$ in \eqref{theoremA1}. {
\begin{theorem}\label{theoremA2}
	Consider the sum-of-exponential approximation of FIDR defined in Theorem \ref{exp}. If we transform the fraction $1/t^\alpha$ into integral with Lemma \ref{exp0}, split the integral interval into $[0,2^{-m}],\ [2^{-m},2^{-m+1}],\ \cdots,\ [2^{n-1},2^n],\ [2^n,\infty]$, apply Gauss-Jacobi quadrature in $[0,2^{-m}]$ with $n_1$ nodes, apply Gauss-Legendre quadrature in $[2^{-m},2^{-m+1}],\ \cdots,\ [2^{n-1},2^n]$ with $n_2$ nodes, and drop the interval $[2^n,\infty]$, then the upper bound of $\e_0$ is,
	\begin{equation}\label{serrorFIDR}
		\begin{split}
			\e_0\leq& \frac{e^{-\delta 2^n}}{\Gamma(\alpha)\delta 2^{(1-\alpha)n}}\\
			&+\frac{1}{\Gamma(\alpha)}\left[\frac{4\sqrt{\pi}2^{-\alpha m}}{e^2}\frac{(2n_1-1)n_1^{3/2}}{(2n_1+\alpha)}\left[\frac{2^{-m}en_1T}{2(2n_1-1)^2}\right]^{2n_1}+2\sqrt{2}\pi 2^{\alpha n}\left(\frac{e^{1/e}}{4}\right)^{2n_2}\right],
		\end{split}
	\end{equation}
	where $n_1,\ n_2$ is the number of nodes in Lemma \ref{exp6}, \ref{exp5}, respectively. The total number of modes $N_A$ the sum-of-exponential approximation of FIDR is $N_A=n_1+(m+n)n_2$.  
\end{theorem}}

{Now we can see the difference of two sum-of-exponential approximations by comparing the two upper bounds of the errors, i.e. \eqref{serrorFIR} and \eqref{serrorFIDR}. When $T=O(1)$ and $\delta=O(\Delta t)$ is very small, the error bound is dominated by the first term. In this scenario, the error bound \eqref{serrorFIDR} is smaller because the first term on the right side of \eqref{serrorFIDR} is smaller. }

{If $\delta=O(\Delta t)$ is not too small, or $\delta$ is small but we choose sufficient large $n$ so that $\frac{e^{-\delta 2^n}}{\Gamma(\alpha)\delta 2^{(1-\alpha)n}}$ is acceptable. Then the error bound in \eqref{serrorFIDR} is dominated by the last term $2\sqrt{2}\pi 2^{\alpha n}\left(\frac{e^{1/e}}{4}\right)^{2n_2}$, since $2^{\alpha n}$ grows exponentially when $n$ becomes larger. In that case, a small $\alpha$ can make the error bound much smaller, because the last term decreases exponentially when $\alpha$ decreases. We summarize the reasoning above as follows.}

{\begin{corollary}\label{FIDR-alpha-error}
	Consider the sum-of-exponential approximation of FIDR in Theorem \ref{theoremA2}. If $T=O(1)$ and $m\ll n$, then the upper bound of $\e_0$ (which is the right side of \eqref{serrorFIDR}) reduces when $\alpha$ gets smaller.
\end{corollary}}

{However, due to the complicated expression of the upper bound, there is no obvious way to give more specific description of the decreasing behavior as $\alpha$ tends to $0$. We shall numerically investigate the dependence of $\varepsilon_0$ on $\alpha$ in Section \ref{num:sum}, and show that the sum-of-exponential approximation of FIDR leads to a better accuracy when compared with its counterpart of FIR.}

\section{Stability and error analysis}\label{Sec-stab}
{In this section, we show the stability and error analysis of Scheme \ref{DFIR} and Scheme \ref{DFIDR} in Section 4.1 and Section 4.2 respectively, for the the initial value problem of the linear fractional diffusion equation \eqref{problem}. 
	We remark that the results in Section 4.1 can be viewed as the improved estimates of those in \cite{JiangZhang}, because they do not quantify the relationship between the global error and $\alpha$. }
\subsection{Stability and error analysis of FIR}
Consider the initial-boundary value problem \eqref{problem}. We recall the full scheme \eqref{DFIR1} and \eqref{dif1} here for convenience.
\begin{equation}
	\begin{aligned}
		& \mathbb{D}_t^\alpha u_i^n=\delta^2_x u_i^n+f_i^n, &1\leq i\leq N_S-1,1\leq n\leq N_T,\\
		&\mathbb{D}_t^\alpha u_0^n=\frac 2h\left[\delta_x u_{\frac 12}^n- \mathbb{D}_t^{\frac \alpha 2} u_0^n\right]+f_0^n,\\
		&\mathbb{D}_t^\alpha u_{N_S}^n=\frac 2 h \left[-\delta_x u_{N_S-\frac 12}^n-  \mathbb{D}_t^{\frac \alpha 2} u_{N_S}^n\right]+f_{N_S}^n,\\
		&u_i^0=u_0(x_i),&0\leq i\leq N_S
	\end{aligned}
\end{equation}\\
{Here $\mathbb{D}_t^\alpha u^n$ is defined in \eqref{DFIR1}, we can write it in following form}
\begin{equation}\label{problem2}
	\begin{aligned}
		\mathbb{D}_t^\alpha u^n=&\frac{\Delta t^{-\alpha}}{\Gamma(1-\alpha)}\left(\frac{u^n}{1-\alpha}-(\frac{\alpha}{1-\alpha}+a_0)u^{n-1}\right.\\
		&\qquad\left.-\sum_{l=1}^{n-2}(a_{n-l-1}+b_{n-l-2})u^l-\left(b_{n-2}+\frac{1}{n^\alpha}\right)u^0\right)
	\end{aligned}
\end{equation}
where
\begin{equation*}
	\begin{aligned}
		&a_n=\alpha\Delta t^\alpha\sum_{j=1}^{N_{\text{exp}}}\omega_j e^{-ns_j\Delta t}\lambda_j^1, &b_n=\alpha \Delta t^\alpha\sum_{j=1}^{N_{\text{exp}}}\omega_j e^{-ns_j\Delta t}\lambda_j^2,\\
		&\lambda_j^1=\frac{e^{-s_j\Delta t}}{s_j^2\Delta t}\left(e^{-s_j\Delta t}-1+s_j\Delta t\right),
		&\lambda_j^2=\frac{e^{-s_j\Delta t}}{s_j^2\Delta t}\left(1-e^{-s_j\Delta t}-e^{-s_j\Delta t}s_j\Delta t\right).
	\end{aligned}
\end{equation*}


{To estimate the global error, we first carry out a prior estimate in Theorem \ref{FIRprior}, then we use the prior estimate to get the upper bound of the local error. Finally we obtain the global error by summing up the local error. }

\begin{theorem}[prior estimate]\label{FIRprior}
	Suppose $\{u_i^k|0\leq i\leq N_S,0\leq k\leq N_T\}$ is the
	solution of the finite difference scheme \eqref{dif1}. Then for any $1\leq n\leq N_T$,
	\begin{equation}\label{theorem1}
		\begin{aligned}
			\Delta t\sum_{k=1}^{n}||u^k||_{\infty}^2\leq&\frac{2\left(1+\sqrt{1+L^2\mu}\right)}{L\mu}\bigg(\rho||u^0||^2+\kappa[(u_0^0)^2+(u_{N_S}^0)^2]\\
			&+\frac{\Delta t}{8\nu}\sum_{k=1}^{n}\left[(hf_0^k)^2+(hf_{N_S}^k)^2\right]+\frac{\Delta t}{\mu}\sum_{k=1}^n h\sum_{i=1}^{N_S-1}(f_i^k)^2\bigg),
		\end{aligned}
	\end{equation} 
	where
	\begin{align*}
		&\rho = \frac{t_n^{1-\alpha}-\alpha(1-\alpha)\e t_{n-1}\Delta t}{\Gamma(2-\alpha)}, &\mu=\frac{t_n^{-\alpha}-2\alpha \e t_{n-1}}{\Gamma(1-\alpha)},\\
		&\kappa =\frac{t_n^{1-\frac \alpha 2}-\frac \alpha 2\left(1-\frac \alpha 2\right)\e t_{n-1}\Delta t}{\Gamma\left(2-\frac \alpha 2\right)},&\nu=\frac{t_n^{-\frac \alpha 2}-\alpha \e t_{n-1}}{\Gamma\left(1-\frac \alpha 2\right)}.
	\end{align*}
\end{theorem}

\begin{proof}
	Multiplying $hu_i^k$ on both sides of the first equation of \eqref{dif1}, and summing up for $i$ from 1 to
	$N_S-1$, we have 
	\begin{align*}
		h\sum_{i=1}^{N_S-1}\left( \mathbb{D}_t^\alpha u_i^k\right)u_i^k-h\sum_{i=1}^{N_S-1}(\delta_x^2 u_i^k)u_i^k=h\sum_{i=1}^{N_S-1}f_i^ku_i^k.
	\end{align*}
	Multiplying $\frac h2 u_0^k$ and $\frac h2 u_{N_S}^k$ on both sides of the second equation and the third equation of \eqref{dif1}, respectively, then
	adding the results with the above identity, we obtain
	\begin{equation}\label{theorem1.1}
		\begin{aligned}
			&\left( \mathbb{D}_t^\alpha u^k,u^k\right)+\left[-\left(\delta_xu_{\frac{1}{2}}^k\right)u_0^k-h\sum_{i=1}^{N_S-1}\left(\delta_x^2 u_i^k\right)u_i^k+\left(\delta_x u_{N_S-\frac 12}^k\right)u_{N_S}^k\right]\\
			&\quad +\left( \mathbb{D}_t^{\frac \alpha 2} u_0^k\right)u_0^k+\left( \mathbb{D}_t^{\frac \alpha 2} u_{N_S}^k\right)u_{N_S}^k=\frac 12 \left(hf_0^k\right)u_0^k+h\sum_{i=1}^{N_S-1}f_i^ku_i^k+\frac 12 \left(hf_{N_S}^k\right)u_{N_S}^k.
		\end{aligned}
	\end{equation}
	Observing the summation by parts, we have
	\begin{equation}\label{theorem1.2}
		\begin{aligned}
			&-\left(\delta_xu_{\frac{1}{2}}^k\right)u_0^k-h\sum_{i=1}^{N_S-1}\left(\delta_x^2 u_i^k\right)u_i^k+\left(\delta_x u_{N_S-\frac 12}^k\right)u_{N_S}^k\\
			=&-\left(\delta_xu_{\frac{1}{2}}^k\right)u_0^k-\sum_{i=1}^{N_S-1}\left(\delta_xu_{i+\frac 12}^k-\delta_xu_{i-\frac 12}^k\right)u_i^k+\left(\delta_x u_{N_S-\frac 12}^k\right)u_{N_S}^k\\
			=&\sum_{i=1}^{N_S}\left(\delta_xu_{i-\frac 12}^k\right)\left(u_i^k-u_{i-1}^k\right)\\
			=&h\sum_{i=1}^{N_S}\left(\delta_xu_{i-\frac 12}^k\right)^2=||\delta_xu^k||^2.
		\end{aligned}
	\end{equation}
	Substituting \eqref{theorem1.2} into \eqref{theorem1.1}, and multiplying $\Delta t$ on both sides of the resulting
	identity, and summing up for k from 1 to n,
	\begin{equation}\label{theorem1.3}
		\begin{aligned}
			&\Delta t\sum_{k=1}^n\left( \mathbb{D}_t^\alpha u^k,u^k\right)+\Delta t\sum_{k=1}^n||\delta_xu^k||^2+\Delta t\sum_{k=1}^n\left( \mathbb{D}_t^{\frac \alpha 2} u_0^k\right)u_0^k+\Delta t\sum_{k=1}^n\left( \mathbb{D}_t^{\frac \alpha 2} u_{N_S}^k\right)u_{N_S}^k\\
			=&\Delta t\sum_{k=1}^n\left[\frac 12 (hf_0^k)u_0^k+h\sum_{i=1}^{N_S-1}f_i^ku_i^k+\frac 12 (hf_{N_S}^k)u_{N_S}^k\right].
		\end{aligned}
	\end{equation}
	Here we need the lemma below.\\
	\begin{lemma}\label{app}
		For any mesh functions $g=\{g^k|0\leq k\leq N \}$ defined on $\Omega_t$, the
		following inequality holds:
		\begin{equation}\label{lemma1}
			\Delta t\sum_{k=1}^n\left( \mathbb{D}_t^\alpha g^k\right)g^k\geq \frac{t_n^{-\alpha}-2\alpha\e t_{n-1}}{2\Gamma(1-\alpha)}\Delta t\sum_{k=1}^n(g^k)^2-\frac{t_n^{1-\alpha}-\alpha(1-\alpha)\e t_{n-1}\Delta t}{\Gamma(2-\alpha)}(g^0)^2.
		\end{equation}
	\end{lemma} 
	We give the proof of this lemma in the appendix.
	
	With this lemma, we get
	\begin{equation*}
		\begin{aligned}
			&\Delta t\sum_{k=1}^n\left( \mathbb{D}_t^\alpha u^k,u^k\right)\geq \frac{t_n^{-\alpha}-2\alpha\e t_{n-1}}{2\Gamma(1-\alpha)}\Delta t\sum_{k=1}^n||u^k||^2-\frac{t_n^{1-\alpha}-\alpha(1-\alpha)\e t_{n-1}\Delta t}{\Gamma(2-\alpha)}||u^0||^2,\\
			&\Delta t\sum_{k=1}^n\left( \mathbb{D}_t^{\frac \alpha 2} u_0^k\right)u_0^k\geq \frac{t_n^{-\frac \alpha 2}-\alpha \e t_{n-1}}{2\Gamma(1-\frac \alpha 2)}\Delta t\sum_{k=1}^n(u_0^k)^2-\frac{t_n^{1-\frac \alpha 2}-\frac \alpha 2(1-\frac \alpha 2)\e t_{n-1}\Delta t}{\Gamma(2-\frac \alpha 2)}(u_0^0)^2,\\
			&\Delta t\sum_{k=1}^n\left(\mathbb{D}_t^{\frac \alpha 2} u_{N_S}^k\right)u_{N_S}^k\geq \frac{t_n^{-\frac \alpha 2}-\alpha \e t_{n-1}}{2\Gamma(1-\frac \alpha 2)}\Delta t\sum_{k=1}^n(u_{N_S}^k)^2-\frac{t_n^{1-\frac \alpha 2}-\frac \alpha 2(1-\frac \alpha 2)\e t_{n-1}\Delta t}{\Gamma(2-\frac \alpha 2)}(u_{N_S}^0)^2.
		\end{aligned}
	\end{equation*}
	Substituting these equations into \eqref{theorem1.3}, we have
	\begin{equation}\label{theorem1.4}
		\begin{aligned}
			&\frac{t_n^{-\alpha}-2\alpha\e t_{n-1}}{2\Gamma(1-\alpha)}\Delta t\sum_{k=1}^n||u^k||^2+\Delta t\sum_{k=1}^n|\delta_xu^k||^2\\
			&\quad+\frac{t_n^{-\frac \alpha 2}-\alpha \e t_{n-1}}{2\Gamma\left(1-\frac \alpha 2\right)}\Delta t\sum_{k=1}^n\left[(u_0^k)^2+(u_{N_S}^k)^2\right]\\
			=&\frac{t_n^{1-\alpha}-\alpha(1-\alpha)\e t_{n-1}\Delta t}{\Gamma(2-\alpha)}||u^0||^2+\frac{t_n^{1-\frac \alpha 2}-\frac \alpha 2\left(1-\frac \alpha 2\right)\e t_{n-1}\Delta t}{\Gamma\left(2-\frac \alpha 2\right)}\left[(u_0^0)^2+(u_{N_S}^0)^2\right]\\
			&\quad+\Delta t\sum_{k=1}^n\left[\frac 12 (hf_0^k)u_0^k+h\sum_{i=1}^{N_S-1}f_i^ku_i^k+\frac 12 (hf_{N_S}^k)u_{N_S}^k\right].
		\end{aligned}
	\end{equation}
	Applying the Cauchy-Schwarz inequality, we obtain
	\begin{equation*}
		\begin{aligned}
			&\frac 12 (hf_0^k)u_0^k+h\sum_{i=1}^{N_S-1}f_i^ku_i^k+\frac 12 (hf_{N_S}^k)u_{N_S}^k\\
			\leq&\frac{t_n^{-\frac \alpha 2}-\alpha\e t_{n-1}}{2\Gamma\left(1-\frac \alpha 2\right)}\left[(u_0^k)^2+(u_{N_S}^k)^2\right]+\frac{\Gamma\left(1-\frac \alpha 2\right)}{8\left(t_n^{-\frac \alpha 2}-\alpha\e t_{n-1}\right)}\left[(hf_0^k)^2+(hf_{N_S}^k)^2\right]\\
			&\quad+h\sum_{i=1}^{N_S-1}\left[\frac{t_n^{-\alpha}-2\alpha\e t_{n-1}}{4\Gamma(1-\alpha)}(u_i^k)^2+\frac{\Gamma(1-\alpha)}{t_n^{-\alpha}-2\alpha\e t_{n-1}}(f_i^k)^2\right]\\
			\leq&\frac{t_n^{-\frac \alpha 2}-\alpha\e t_{n-1}}{2\Gamma\left(1-\frac \alpha 2\right)}\left[(u_0^k)^2+(u_{N_S}^k)^2\right]+\frac{\Gamma(1-\frac \alpha 2)}{8\left(t_n^{-\frac \alpha 2}-\alpha\e t_{n-1}\right)}\left[(hf_0^k)^2+(hf_{N_S}^k)^2\right]\\
			&\quad+\frac{t_n^{-\alpha}-2\alpha\e t_{n-1}}{4\Gamma(1-\alpha)}||u^k||^2+h\sum_{i=1}^{N_S-1}\frac{\Gamma(1-\alpha)}{t_n^{-\alpha}-2\alpha\e t_{n-1}}(f_i^k)^2.\\
		\end{aligned}
	\end{equation*}
	The substitution of this equation into \eqref{theorem1.4} produces
	\begin{equation}\label{theorem1.5}
		\begin{aligned}
			&\frac \mu 4 \Delta t\sum_{k=1}^{n}||u^k||^2+\Delta t\sum_{k=1}^{n}||\delta_x u^k||^2\leq \rho ||u^0||^2+\kappa\left[(u_0^0)^2+(u_{N_S}^0)^2\right]\\
			&  \qquad\qquad +\frac{\Delta t}{8\nu}\sum_{k=1}^n\left[(hf_0^k)^2+(hf_{N_S}^k)^2\right]+\frac{\Delta t}{\mu}\sum_{k=1}^n h\sum_{i=1}^{N_S-1}(f_i^k)^2. 
		\end{aligned}
	\end{equation}
\end{proof}
Now we need to use the following lemma to bound the global error.\\
\begin{lemma}\label{mesh}
	For any mesh function u defined on $S_h=\{u|u=(u_0,u_1,\cdots,u_{N_S})\}$, the following inequality holds
	\begin{equation}\label{lemma2}
		||u||_{\infty}^2\leq \theta ||\delta_x u||^2+(\frac{1}{\theta}+\frac 1 L)||u||^2,\quad \forall \theta>0.
	\end{equation}
	where $L$ is the length of the computational domain and here, $L=x_r-x_l$.
\end{lemma}
Taking $\theta >0$ such that $\frac{1/\theta+1/L}{\theta}=\frac \mu 4$ (i.e. $\theta=2(1+\sqrt{1+L^2\mu})/(L\mu)$), and following
from Lemma \ref{mesh}, we have 
\begin{equation}\label{theorem1.6}
	\Delta t\sum_{k=1}^n||u^k||_{\infty}^2\leq \frac{2\left(1+\sqrt{1+L^2\mu}\right)}{L\mu}\left(\frac \mu 4 \Delta t\sum_{k=1}^{n}||u^k||^2+\Delta t\sum_{k=1}^{n}||\delta_x u^k||^2\right).
\end{equation}
Combining \eqref{theorem1.6} and \eqref{theorem1.5}, we obtain the inequality \eqref{theorem1} \\

The priori estimate leads to the stability of the FIR scheme. Now we present an error analysis of the scheme.
\begin{theorem}[Error Analysis]\label{FIRerror2}
	Suppose $u(x,t)\in C_{x,t}^{4,2}([x_l,x_r]\times [0,T]) $ and $\{u_i^k|0\leq i\leq N_S,\ 0\leq k\leq N_T\}$ are solutions of the problem \eqref{problem} and the
	difference scheme \eqref{DFIR1} and \eqref{dif1}, respectively. Let $e_i^k=u_i^k-u(x_i,t_k)$. Then there
	exists a positive constant $c_2$ such that 
	\begin{equation}\label{theorem2}
		\e_{\text{global}}:=\sqrt{\Delta t \sum_{k=1}^n||e^k||_{\infty}^2}\leq c_2(h^2+\Delta t^{2-\alpha}+{\alpha}\e),\qquad 1\leq n\leq N_T,
	\end{equation}
	where $c_2^2=\frac{4c_1^2T\left(1+\sqrt{1+L^2\mu}\right)}{L\mu}\left(\frac 1\nu+\frac{L}{\mu}\right)$ with $c_1$ is a positive constant, and $\mu, \nu$ are defined in theorem \ref{FIRprior}.
\end{theorem}

\begin{proof}
	We observe that the error $e_i^k$ satisfies the following FD scheme:
	\begin{equation}\label{dif2}
		\begin{aligned}
			\mathbb{D}_t^\alpha e_i^k&=\delta^2_x e_i^k+T_i^k, &1\leq i\leq N_S-1,1\leq k\leq N_T,\\
			\mathbb{D}_t^\alpha e_0^k&=\frac 2h\left[\delta_x e_{\frac 12}^k- \mathbb{D}_t^{\frac \alpha 2} e_0^k\right]+T_0^k,\\
			\mathbb{D}_t^\alpha e_{N_S}^k&=\frac 2 h \left[-\delta_x e_{N_S-\frac 12}^k-  \mathbb{D}_t^{\frac \alpha 2} e_{N_S}^k\right]+T_{N_S}^k,\\
			e_i^0&=0,&0\leq i\leq N_S,
		\end{aligned}
	\end{equation}
	where the truncation terms $T^k$ at the interior and boundary points are given by the
	formulas
	\begin{equation*}
		\begin{aligned}
			T_i^k=&-\left[ ^C_0D_t^\alpha u(x_i,t_k)- \mathbb{D}_t^\alpha U_i^k\right]+\left[u_{xx}(x_i,t_k)-\delta_x^2U_i^k\right], \quad 1\leq i\leq N_S-1,1\leq k\leq N_T,\\
			T_0^k=&\left\lbrace u_{xx}(x_0,t_k)-\frac 2 h\left[\delta_xU_{\frac 12}^k-u_x(x_0,t_k)\right]-\frac 2 h \left[ ^C_0D_t^{\frac \alpha 2}u(x_0,t_k)- \mathbb{D}_t^{\frac \alpha 2}U_0^k\right] \right\rbrace\\
			&\quad -\left[ ^C_0D_t^\alpha u(x_0,t_k)- \mathbb{D}_t^\alpha U_0^k\right],\\
			T_{N_S}^k=&\left\lbrace u_{xx}(x_{N_S},t_k)+\frac 2 h\left[\delta_xU_{N_S-\frac 12}^k-u_x(x_{N_S},t_k)\right]-\frac 2 h \left[ ^C_0D_t^{\frac \alpha 2}u(x_{N_S},t_k)- \mathbb{D}_t^{\frac \alpha 2}U_{N_S}^k\right] \right\rbrace\\
			&\quad -\left[ ^C_0D_t^\alpha u(x_{N_S},t_k)- \mathbb{D}_t^\alpha U_{N_S}^k\right].\\
		\end{aligned}
	\end{equation*}
	We can show that the truncation terms $T^k$ satisfy the following error bounds
	\begin{equation}\label{theorem2.1}
		\begin{aligned}
			|T_i^k|\leq &c_1\left(\Delta t^{2-\alpha}+h^2+\alpha\e\right),\\
			|T_0^k|\leq& c_1\left(\Delta t^{2-\alpha}+h+\frac{\Delta t^{2-\alpha/2}}{h}+\frac {\alpha\e} h\right),\\
			|T_{N_S}^k|\leq& c_1\left(\Delta t^{2-\alpha}+h+\frac{\Delta t^{2-\alpha/2}}{h}+\frac {\alpha\e} h\right).
		\end{aligned}
	\end{equation}
	with $c_1$ some positive constant. Thus, for $h\leq 1$ and $\Delta t\leq 1$, we have
	\begin{equation}\label{theorem2.2}
		\begin{aligned}
			&\frac{1}{4\nu}\left[(hT_0^k)^2+(hT_{N_S}^k)^2\right]+\frac 2\mu h\sum_{i=1}^{N_S-1}(T_i^k)^2\\
			\leq&\frac{c_1^2}{2\nu}\left(h\Delta t^{2-\alpha}+\Delta t^{2-\frac \alpha 2}+h^2+\alpha\e\right)^2+\frac{2c_1^2L}{\mu}(\Delta t^{2-\alpha}+h^2+\alpha\e)^2\\
			\leq& \frac{2c_1^2}{\nu}\left(\Delta t^{2-\alpha}+h^2+\alpha\e\right)^2+\frac{2c_1^2L}{\mu}\left(\Delta t^{2-\alpha}+h^2+\alpha\e\right)^2\\
			=&\left(\frac{2c_1^2}{\nu}+\frac{2c_1^2L}{\mu}\right)\left(\Delta t^{2-\alpha}+h^2+\alpha\e\right)^2.
		\end{aligned}
	\end{equation}
	A direct application of theorem \ref{FIRprior} to \eqref{dif2} produces
	\begin{equation}\label{theorem2.3}
		\begin{aligned}
			\Delta t\sum_{k=1}^n||e^k||_{\infty}^2&\leq \frac{\Delta t\left(1+\sqrt{1+L^2\mu}\right)}{L\mu}\sum_{k=1}^n\left(\frac{1}{4\nu}\left[(hT_0^k)^2+(hT_{N_S}^k)^2\right]+\frac 2\mu h\sum_{i=1}^{N_S-1}(T_i^k)^2\right)\\
			&\leq\frac{T\left(1+\sqrt{1+L^2\mu}\right)}{L\mu}\left(\frac{2c_1^2}{\nu}+\frac{2c_1^2L}{\mu}\right)\left(\Delta t^{2-\alpha}+h^2+\alpha\e\right)^2.
		\end{aligned}
	\end{equation}
	Finally we obtain \eqref{theorem2}
\end{proof}

\subsection{Stability and error analysis of FIDR}
In order to estimate the global error of the FIDR scheme applying to the linear fraction diffusion problem, we can carry out a similar analysis routine as in the previous section. However, the consistency of the FIDR approximation to the Caputo derivative is lacking prior to this work. We show in the following that the consistency error of the FIDR approximation can be bounded with the consistency error of the L1-approximation, which is stated in Theorem \ref{boundFIDR}. Thus, the rest of the error analysis follows naturally. 


{The following theorem, which can be found in \cite[Lemma 4.1]{sun2006fully}, establishes an error bound for the L1-approximation.}
\begin{theorem}
	(see \cite{sun2006fully}) Suppose that $u(t)\in C^2[0,t_n],\ ^C_0\mathbb{D}_t^\alpha$ is the L1 approximation of Caputo derivative, and let
	\begin{equation}
		R^nu:=\ ^C_0D_t^\alpha u(t)|_{t=t_n} - ^C_0\mathbb{D}_t^\alpha u^n,
	\end{equation}
	where $0<\alpha<1$. Then
	\begin{equation}\label{errorbound1}
		|R^nu|\leq \frac{\Delta t^{2-\alpha}}{\Gamma(2-\alpha)}\left(\frac{1-\alpha}{12}+\frac{2^{2-\alpha}}{2-\alpha}-(1+2^{-\alpha})\right)\max_{0\leq t\leq t_n}|u''(t)|.
	\end{equation}
\end{theorem} 
The following theorem provides an error bound for our approximation.

\begin{theorem}\label{boundFIDR}
	Suppose that $u(t)\in C^2[0,t_n],\ \mathbb{D}_t^\alpha$ is the approximation in \eqref{DFIDR1}, and let 
	\begin{equation}
		^{F2}R^nu:=\ ^C_0D_t^\alpha u(t)|_{t=t_n} - \mathbb{D}_t^\alpha u^n,
	\end{equation}
	where $0<\alpha<1$. Then
	\begin{multline}\label{errorbound2}
		| ^{F2}R^nu|\leq \frac{\Delta t^{2-\alpha}}{\Gamma(2-\alpha)}\left(\frac{1-\alpha}{12}+\frac{2^{2-\alpha}}{2-\alpha}-(1+2^{-\alpha})\right)\max_{0\leq t\leq t_n}|u''(t)|
	\\	+\frac{\e_0t_{n-1}}{\Gamma(1-\alpha)}\max_{0\leq t\leq t_{n-1}}|u'(t)|.
	\end{multline}
\end{theorem} 

\begin{proof}
	the only difference between our approximation $\mathbb{D}_t^\alpha u^n$ and the L1-approximation $^C_0\mathbb{D}_t^\alpha u^n$ is that the convolution kernel admits an absolute error bounded by $\e_0$ in its sum-of-exponentials approximation \eqref{theoremA1}, thus
	\begin{equation}\label{theoremB3}
		\begin{aligned}
			| \mathbb{D}_t^\alpha u^n- ^C_0\mathbb{D}_t^\alpha u^n|\leq \frac{\e_0}{\Gamma(1-\alpha)}\sum_{l=1}^{n-1}\int_{t_{l-1}}^{t_l}|\Pi_{1,l}u(s)|ds.
		\end{aligned}
	\end{equation}
	where $\Pi_{1,l}u(s)$ is the approximation of $u'(s)$, $s\in [t_{l-1},t_l]$ used in L1 approximation and our approximation,
	\begin{equation*}
		\begin{aligned}
			\Pi_{1,l}u(s)=\frac{u^l-u^{l-1}}{\Delta t}.
		\end{aligned}
	\end{equation*}
	Thus
	\begin{equation}
		\sum_{l=1}^{n-1}\int_{t_{l-1}}^{t_l}|\Pi_{1,l}u(s)|ds\leq\max_{0\leq t\leq t_{n-1}}|u'(t)|t_{n-1},
	\end{equation}
	take it into \eqref{theoremB3}, and the triangle inequality leads to
	\begin{equation}
		\begin{aligned}
			|^{F2}R^nu|&\leq |R^nu|+\frac{\e_0}{\Gamma(1-\alpha)}\sum_{l=1}^{n-1}\int_{t_{l-1}}^{t_l}|\Pi_{1,l}u(s)|ds\\
			&\leq \frac{\Delta t^{2-\alpha}}{\Gamma(2-\alpha)}\left(\frac{1-\alpha}{12}+\frac{2^{2-\alpha}}{2-\alpha}-(1+2^{-\alpha})\right)\max_{0\leq t\leq t_n}|u''(t)|\\
			& \quad  +\frac{\e_0 t_{n-1}}{\Gamma(1-\alpha)}\max_{0\leq t\leq t_{n-1}}|u'(t)|.
		\end{aligned}
	\end{equation}
	We obtain the theorem. 
\end{proof}

To calculate the global error, we first need the expanded form of $\mathbb{D}_t^\alpha $.
\begin{equation}
	\begin{aligned}
		\mathbb{D}_t^\alpha u^n=&\frac{u^n-u^{n-1}}{\Delta t^\alpha\Gamma(2-\alpha)}+ \frac{1}{\Gamma(1-\alpha)}\left(\sum_{i=1}^{N_{A}}\tilde{w}_i\tilde{\psi}(t_{n},\tilde{s}_i)\right)\\
		=&\frac{u^n-u^{n-1}}{\Delta t^\alpha\Gamma(2-\alpha)}+\frac{1}{\Gamma(1-\alpha)}\left[\sum_{i=1}^{N_{A}}\tilde{w}_i\right.\\
		&\quad \left.\left(e^{-\tilde{s}_i\Delta t}\psi(t_{n-1},\tilde{s}_i)+\frac{(u^{n-1}-u^{n-2})\left(1-e^{-\tilde{s}_i\Delta t}\right)e^{-\tilde{s}_i\Delta t}}{\tilde{s}_i\Delta t}\right)\right]\\		=&\frac{u^n-u^{n-1}}{(1-\alpha)\Delta t^\alpha\Gamma(1-\alpha)}\\
		&\quad+\frac{1}{\Gamma(1-\alpha)}\left[a_1 u^{n-1}+\sum_{l=2}^{n-1}(a_l-a_{l-1})u^{n-l}-a_{n-1}u^0\right]\\
		=&\frac{1}{\Gamma(1-\alpha)}\left[\frac{u^n}{(1-\alpha)\Delta t^{\alpha}}\right.\\
		&\quad\left.+\left(a_1-\frac{1}{(1-\alpha)\Delta t^{\alpha}}\right)u^{n-1}+\sum_{l=2}^{n-1}(a_l-a_{l-1})u^{n-l}-a_{n-1}u^0\right],
	\end{aligned}
\end{equation}
where
\begin{equation*}
	\begin{aligned}
		a_l=\sum_{i=1}^{N_A}\frac{\tilde{w}_i(1-e^{-\tilde{s}_i\Delta t})e^{-l\tilde{s}_i\Delta t}}{\tilde{s}_i\Delta t}.
	\end{aligned}
\end{equation*}

{To do the prior estimate, first we need to prove the following lemma.}
\begin{lemma}
	For any mesh functions $u=\{u^k|0\leq k\leq N \}$ defined on $\Omega_t$, the
	following inequality holds:
	\begin{equation}\label{LemmaB4}
		\begin{aligned}
			\Delta t\sum_{k=1}^{n}(\mathbb{D}_t^\alpha u^k)u^k&\geq\frac{\Delta t(t_n^{-\alpha}-\e_0)}{2\Gamma(1-\alpha)}\sum_{k=1}^n(u^k)^2-\frac{\Delta t^{1-\alpha}(1-\alpha)^{-1}+t_{n-1}\Delta t^{-\alpha}}{2\Gamma(1-\alpha)}(u^0)^2.
		\end{aligned}
	\end{equation}
\end{lemma}

\begin{proof}
	cause $\tilde{s}_i>0$, we have $a_l-a_{l-1}<0$, and
	\begin{equation*}
		\begin{aligned}
			a_1&=\sum_{i=1}^{N_A}\frac{\tilde{w}_i\left(1-e^{-\tilde{s}_i\Delta t}\right)e^{-\tilde{s}_i\Delta t}}{\tilde{s}_i\Delta t}\\
			&\leq \left(\sum_{i=1}^{N_A}\tilde{w}_ie^{-\tilde{s}_i\Delta t}\right)\max_{x>0}\frac{1-e^{-x}}{x}\\
			&\leq\left(\frac{1}{\Delta t^\alpha}+\e_0\right)\max_{x>0}\frac{1-e^{-x}}{x}\\
			&\leq\frac{1}{\Delta t^\alpha}+\e_0\leq\frac{1}{(1-\alpha)\Delta t^{\alpha}}.
		\end{aligned}
	\end{equation*}
	
	Thus, with Cauchy-Schwarz inequality, we obtain
	\begin{equation*}
		\begin{aligned}
			(\mathbb{D}_t^\alpha u^n)u^n=&\frac{1}{\Gamma(1-\alpha)}\left[\frac{(u^n)^2}{(1-\alpha)\Delta t^{\alpha}}+\left(a_1-\frac{1}{(1-\alpha)\Delta t^{\alpha}}\right)u^nu^{n-1}\right.\\
			&\quad\left.+\sum_{l=2}^{n-1}(a_l-a_{l-1})u^nu^{n-l}-a_{n-1}u^nu^0\right]\\
			\geq& \frac{1}{\Gamma(1-\alpha)}\left\{\left[\frac{1}{(1-\alpha)\Delta t^{\alpha}}-\frac 12 \left(\frac{1}{(1-\alpha)\Delta t^{\alpha}}-a_1\right)-\frac 12\sum_{l=2}^{n-1}(a_{l-1}-a_{l})\right.\right.\\
			&\quad\left. -\frac 12 a_{n-1}\right](u^n)^2-\frac 12\left(\frac{1}{(1-\alpha)\Delta t^{\alpha}}-a_1\right)(u^{n-1})^2\\
			&\quad\left.-\frac 12\sum_{l=2}^{n-1}(a_{l-1}-a_{l})(u^{n-l})^2-\frac 12 a_{n-1}(u^0)^2  \right\}\\
			=&\frac{1}{\Gamma(1-\alpha)}\left[\frac 12\frac{1}{(1-\alpha)\Delta t^{\alpha}}(u^n)^2-\frac 12\left(\frac{1}{(1-\alpha)\Delta t^{\alpha}}-a_1\right)(u^{n-1})^2\right.\\
			&\quad \left.-\frac 12\sum_{l=2}^{n-1}(a_{l-1}-a_{l})(u^{n-l})^2-\frac 12 a_{n-1}(u^0)^2  \right].\\
		\end{aligned}
	\end{equation*}
	
	Summing the above inequality from 1 to n, we obtain
	\begin{equation}\label{theoremb4}
		\begin{aligned}
			\Delta t\sum_{k=1}^{n}\left( \mathbb{D}_t^\alpha u^k\right)u^k&\geq\frac{\Delta t}{\Gamma(1-\alpha)}\left[\frac 12\frac{1}{(1-\alpha)\Delta t^{\alpha}}(u^n)^2+\sum_{l=1}^{n-1}\frac 12 a_l (u^{n-l})^2\right.\\
			&\quad \left.-\left(\frac 12\frac{1}{(1-\alpha)\Delta t^{\alpha}}+\sum_{l=1}^{n-1}\frac 12 a_l\right)(u^0)^2\right]\\
			&\geq\frac{\Delta t}{\Gamma(1-\alpha)}\left[\frac 12 a_{n-1}\sum_{k=1}^n(u^k)^2-\left(\frac 12\frac{1}{(1-\alpha)\Delta t^{\alpha}}+\sum_{l=1}^{n-1}\frac 12 a_l\right)(u^0)^2 \right],
		\end{aligned}
	\end{equation}
	where
	\begin{equation*}
		\begin{aligned}
			a_{n-1}&=\sum_{i=1}^{N_A}\frac{\tilde{w}_i(1-e^{-\tilde{s}_i\Delta t})e^{-(n-1)\tilde{s}_i\Delta t}}{\tilde{s}_i\Delta t}\\
			&\geq \left(\sum_{i=1}^{N_A}\tilde{w}_ie^{-n\tilde{s}_i\Delta t}\right)\min_{x>0}\frac{e^{x}-1}{x}\\
			&\geq \left(\frac{1}{t_n^{\alpha}}-\e_0\right)\min_{x>0}\frac{e^{x}-1}{x}\\
			&\geq \frac{1}{t_n^{\alpha}}-\e_0,
		\end{aligned}
	\end{equation*}
	and
	\begin{equation*}
		\begin{aligned}
			\frac{1}{(1-\alpha)\Delta t^{\alpha}}+\sum_{l=1}^{n-1} a_l&=\frac{1}{(1-\alpha)\Delta t^{\alpha}}+ \sum_{i=1}^{N_A}\sum_{l=1}^{n-1}\frac{\tilde{w}_i(1-e^{-\tilde{s}_i\Delta t})e^{-l\tilde{s}_i\Delta t}}{\tilde{s}_i\Delta t}\\
			&=\frac{1}{(1-\alpha)\Delta t^{\alpha}}+ \sum_{i=1}^{N_A}\frac{\tilde{w}_i(e^{-\tilde{s}_i\Delta t}-e^{-n\tilde{s}_i\Delta t})}{\tilde{s}_i\Delta t}\\
			&\leq\frac{1}{(1-\alpha)\Delta t^{\alpha}}+ \left(\sum_{i=1}^{N_A}\tilde{w}_ie^{-\tilde{s}_i\Delta t}\right)\max_{x>0}\frac{1-e^{-(n-1)x}}{x}\\
			&\leq \frac{1}{(1-\alpha)\Delta t^{\alpha}}+(n-1)\frac{1}{\Delta t^{\alpha}}\\
			&=\frac{1+(1-\alpha)(n-1)}{(1-\alpha)\Delta t^{\alpha}}.
		\end{aligned}
	\end{equation*}
	Substitute these equations into \eqref{theoremb4}, we get
	\begin{equation}
		\begin{aligned}
			\Delta t\sum_{k=1}^{n}\left( \mathbb{D}_t^\alpha u^k\right)u^k\geq &\frac{\Delta t}{\Gamma(1-\alpha)}\left[\frac 12 \left(\frac{1}{t_n^{\alpha}}-\e_0\right)\sum_{k=1}^n(u^k)^2\right.\\
			&\quad\left.-\frac 12 \left(\frac{1+(1-\alpha)(n-1)}{(1-\alpha)\Delta t^{\alpha}}\right)(u^0)^2 \right]\\
			=&\frac{\Delta t(t_n^{-\alpha}-\e_0)}{2\Gamma(1-\alpha)}\sum_{k=1}^n(u^k)^2-\frac{\Delta t^{1-\alpha}(1-\alpha)^{-1}+t_{n-1}\Delta t^{-\alpha}}{2\Gamma(1-\alpha)}(u^0)^2.
		\end{aligned}
	\end{equation}
	And the lemma is proved.
\end{proof}
With this lemma, we can do the prior estimate and calculate the global error like FIR.

Consider the diffusion PDE problem
\begin{equation}\label{problem3}
	\begin{aligned}
		& ^C_0D_t^\alpha u(x,t)=u_{xx}(x,t)+f(x,t), &x\in\Omega,\ t>0,\\
		& u(x,0)=x_0(x),&x\in \Omega,\\
		& \frac{\partial u(x,t)}{\partial x}=\frac{1}{\Gamma (1-\frac \alpha 2)}\int_0^t \frac{u_s(x,s)}{(t-s)^{\frac \alpha 2}}ds:=\ ^C_0D_t^{\frac \alpha 2} u(x,t), &x=x_l,\\
		& \frac{\partial u(x,t)}{\partial x}=-\frac{1}{\Gamma (1-\frac \alpha 2)}\int_0^t \frac{u_s(x,s)}{(t-s)^{\frac \alpha 2}}ds:=- ^C_0D_t^{\frac \alpha 2} u(x,t), &x=x_r.\\
	\end{aligned}
\end{equation}\\
The finite difference scheme 
\begin{equation}\label{dif3}
	\begin{aligned}
		& \mathbb{D}_t^\alpha u_i^n=\delta^2_x u_i^n+f_i^n, &1\leq i\leq N_S-1,1\leq n\leq N_T,\\
		& \mathbb{D}_t^\alpha u_0^n=\frac 2h\left[\delta_x u_{\frac 12}^n-  \mathbb{D}_t^{\frac \alpha 2} u_0^n\right]+f_0^n,\\
		& \mathbb{D}_t^\alpha u_{N_S}^n=\frac 2 h \left[-\delta_x u_{N_S-\frac 12}^n-   \mathbb{D}_t^{\frac \alpha 2} u_{N_S}^n\right]+f_{N_S}^n,\\
		&u_i^0=u_0(x_i),&0\leq i\leq N_S.
	\end{aligned}
\end{equation}

\begin{theorem}[Prior Estimate]\label{FIDRprior}
	Suppose $\{u_i^k|0\leq k\leq N_T,\ 0\leq i\leq N_S\}$ is the
	solution of the finite difference scheme \eqref{dif3}. Then for any $1\leq n\leq N_T$,
	\begin{equation}\label{theoremB5}
		\begin{aligned}
			\Delta t\sum_{k=1}^{n}||u^k||_{\infty}^2\leq&\frac{2\left(1+\sqrt{1+L^2\tilde{\mu}}\right)}{L\tilde{\mu}}\Bigg(\tilde{\rho}||u^0||^2+\tilde{\kappa}[(u_0^0)^2+(u_{N_S}^0)^2]\\
			&+\frac{\Delta t}{8\tilde{\nu}}\sum_{k=1}^{n}[(hf_0^k)^2+(hf_{N_S}^k)^2]+\frac{\Delta t}{\tilde{\mu}}\sum_{k=1}^n h\sum_{i=1}^{N_S-1}(f_i^k)^2\Bigg),
		\end{aligned}
	\end{equation}
	where
	\begin{equation*}
		\begin{aligned}
			\tilde{\rho}&=\frac{\Delta t^{1-\alpha}(1-\alpha)^{-1}+t_{n-1}\Delta t^{-\alpha}}{2\Gamma(1-\alpha)},&\tilde{\kappa}=\frac{\Delta t^{1-\frac \alpha 2}\left(1-\frac \alpha 2\right)^{-1}+t_{n-1}\Delta t^{-\frac \alpha 2}}{2\Gamma\left(1-\frac \alpha 2\right)},\\
			\tilde{\mu}&=\frac{t_n^{-\alpha}-\e_0}{\Gamma(1-\alpha)}, &\tilde{\nu}=\frac{t_n^{-\frac \alpha 2}-\e_0}{\Gamma\left(1-\frac \alpha 2\right)}.\\
		\end{aligned}
	\end{equation*}
\end{theorem}

\begin{proof}
	Multiplying $hu_i^k$ on both sides of the first equation of \eqref{dif3}, and summing up for $i$ from 1 to
	$N_S-1$, we have 
	\begin{align*}
		h\sum_{i=1}^{N_S-1}( \mathbb{D}_t^\alpha u_i^k)u_i^k-h\sum_{i=1}^{N_S-1}(\delta_x^2 u_i^k)u_i^k=h\sum_{i=1}^{N_S-1}f_i^ku_i^k.
	\end{align*}
	Multiplying $\frac h2 u_0^k$ and $\frac h2 u_{N_S}^k$ on both sides of the second equation and the third equation of \eqref{dif3}, respectively, then
	adding the results with the above identity, we obtain
	\begin{equation}\label{sumx}
		\begin{aligned}
			&\left( \mathbb{D}_t^\alpha u^k,u^k\right)+\left[-\left(\delta_xu_{\frac{1}{2}}^k\right)u_0^k-h\sum_{i=1}^{N_S-1}\left(\delta_x^2 u_i^k\right)u_i^k+\left(\delta_x u_{N_S-\frac 12}^k\right)u_{N_S}^k\right]\\
			&\quad +\left( \mathbb{D}_t^{\frac \alpha 2} u_0^k\right)u_0^k+\left( \mathbb{D}_t^{\frac \alpha 2} u_{N_S}^k\right)u_{N_S}^k=\frac 12 (hf_0^k)u_0^k+h\sum_{i=1}^{N_S-1}f_i^ku_i^k+\frac 12 (hf_{N_S}^k)u_{N_S}^k.
		\end{aligned}
	\end{equation}
	Observing the summation by parts, we have
	\begin{equation}\label{eq47}
		\begin{aligned}
			&-\left(\delta_xu_{\frac{1}{2}}^k\right)u_0^k-h\sum_{i=1}^{N_S-1}\left(\delta_x^2 u_i^k\right)u_i^k+\left(\delta_x u_{N_S-\frac 12}^k\right)u_{N_S}^k\\
			=&-\left(\delta_xu_{\frac{1}{2}}^k\right)u_0^k-\sum_{i=1}^{N_S-1}\left(\delta_xu_{i+\frac 12}^k-\delta_xu_{i-\frac 12}^k\right)u_i^k+\left(\delta_x u_{N_S-\frac 12}^k\right)u_{N_S}^k\\
			=&\sum_{i=1}^{N_S}\left(\delta_xu_{i-\frac 12}^k\right)\left(u_i^k-u_{i-1}^k\right)\\
			=&h\sum_{i=1}^{N_S}\left(\delta_xu_{i-\frac 12}^k\right)^2=||\delta_xu^k||^2.
		\end{aligned}
	\end{equation}
	Substituting \eqref{eq47} into \eqref{sumx}, and multiplying $\Delta t$ on both sides of the resulting
	identity, and summing up for k from 1 to n,
	\begin{equation}\label{theoremb51}
		\begin{aligned}
			&\Delta t\sum_{k=1}^n\left( \mathbb{D}_t^\alpha u^k,u^k\right)+\Delta t\sum_{k=1}^n||\delta_xu^k||^2+\Delta t\sum_{k=1}^n\left( \mathbb{D}_t^{\frac \alpha 2} u_0^k\right)u_0^k+\Delta t\sum_{k=1}^n\left( \mathbb{D}_t^{\frac \alpha 2} u_{N_S}^k\right)u_{N_S}^k\\
			=&\Delta t\sum_{k=1}^n\left[\frac 12 \left(hf_0^k\right)u_0^k+h\sum_{i=1}^{N_S-1}f_i^ku_i^k+\frac 12 \left(hf_{N_S}^k\right)u_{N_S}^k\right].
		\end{aligned}
	\end{equation}
	From \eqref{LemmaB4}, we have
	\begin{equation*}
		\begin{aligned}
			\Delta t\sum_{k=1}^n\left( \mathbb{D}_t^\alpha u^k,u^k\right)&\geq\frac{\Delta t(t_n^{-\alpha}-\e_0)}{2\Gamma(1-\alpha)}\sum_{k=1}^n||u^k||^2-\frac{\Delta t^{1-\alpha}(1-\alpha)^{-1}+t_{n-1}\Delta t^{-\alpha}}{2\Gamma(1-\alpha)}||u^0||^2\\
			\Delta t\sum_{k=1}^n\left( \mathbb{D}_t^{\frac \alpha 2} u_0^k\right)u_0^k&\geq\frac{\Delta t\left(t_n^{-\frac \alpha 2}-\e_0\right)}{2\Gamma\left(1-\frac \alpha 2\right)}\sum_{k=1}^n(u_0^k)^2-\frac{\Delta t^{1-\frac \alpha 2}\left(1-\frac \alpha 2\right)^{-1}+t_{n-1}\Delta t^{-\frac \alpha 2}}{2\Gamma\left(1-\frac \alpha 2\right)}(u_0^0)^2\\
			\Delta t\sum_{k=1}^n\left( \mathbb{D}_t^{\frac \alpha 2} u_{N_S}^k\right)u_{N_S}^k&\geq\frac{\Delta t\left(t_n^{-\frac \alpha 2}-\e_0\right)}{2\Gamma\left(1-\frac \alpha 2\right)}\sum_{k=1}^n(u_{N_S}^k)^2-\frac{\Delta t^{1-\frac \alpha 2}\left(1-\frac \alpha 2\right)^{-1}+t_{n-1}\Delta t^{-\frac \alpha 2}}{2\Gamma\left(1-\frac \alpha 2\right)}(u_{N_S}^0)^2.
		\end{aligned}
	\end{equation*}
	Substitute these equation into \eqref{theoremb51}, 
	\begin{equation}\label{theoremb52}
		\begin{aligned}
			&\frac{\Delta t(t_n^{-\alpha}-\e_0)}{2\Gamma(1-\alpha)}\sum_{k=1}^n||u^k||^2+\Delta t\sum_{k=1}^n||\delta_xu^k||^2+\frac{\Delta t\left(t_n^{-\frac \alpha 2}-\e_0\right)}{2\Gamma\left(1-\frac \alpha 2\right)}\sum_{k=1}^n\left[\left(u_0^k\right)^2+\left(u_{N_S}^k\right)^2\right]\\
			\leq&\frac{\Delta t^{1-\alpha}(1-\alpha)^{-1}+t_{n-1}\Delta t^{-\alpha}}{2\Gamma(1-\alpha)}||u^0||^2\\
			&\quad+\frac{\Delta t^{1-\frac \alpha 2}\left(1-\frac \alpha 2\right)^{-1}+t_{n-1}\Delta t^{-\frac \alpha 2}}{2\Gamma\left(1-\frac \alpha 2\right)}\left[\left(u_0^0\right)^2+\left(u_{N_S}^k\right)^2\right]\\
			&\quad+\Delta t\sum_{k=1}^n\left[\frac 12 (hf_0^k)u_0^k+h\sum_{i=1}^{N_S-1}f_i^ku_i^k+\frac 12 \left(hf_{N_S}^k\right)u_{N_S}^k\right].
		\end{aligned}
	\end{equation}
	Applying the Cauchy-Schwarz inequality, we obtain
	\begin{equation*}
		\begin{aligned}
			&\frac 12 \left(hf_0^k\right)u_0^k+h\sum_{i=1}^{N_S-1}f_i^ku_i^k+\frac 12 \left(hf_{N_S}^k\right)u_{N_S}^k\\
			\leq&\frac{t_n^{-\frac \alpha 2}-\e_0}{2\Gamma\left(1-\frac \alpha 2\right)}\left[(u_0^k)^2+\left(u_{N_S}^k\right)^2\right]+\frac{\Gamma(1-\frac \alpha 2)}{8(t_n^{-\frac \alpha 2}-\e_0)}\left[(hf_0^k)^2+\left(hf_{N_S}^k\right)^2\right]\\
			&\quad+h\sum_{i=1}^{N_S-1}\left[\frac{t_n^{-\alpha}-\e_0}{4\Gamma(1-\alpha)}(u_i^k)^2+\frac{\Gamma(1-\alpha)}{t_n^{-\alpha}-\e_0}(f_i^k)^2\right]\\
			\leq&\frac{t_n^{-\frac \alpha 2}-\e_0}{2\Gamma\left(1-\frac \alpha 2\right)}\left[(u_0^k)^2+\left(u_{N_S}^k\right)^2\right]+\frac{\Gamma\left(1-\frac \alpha 2\right)}{8 \left(t_n^{-\frac \alpha 2}-\e_0\right)}\left[(hf_0^k)^2+\left(hf_{N_S}^k\right)^2\right]\\
			&\quad+\frac{t_n^{-\alpha}-\e_0}{4\Gamma(1-\alpha)}||u^k||^2+h\sum_{i=1}^{N_S-1}\frac{\Gamma(1-\alpha)}{t_n^{-\alpha}-\e_0}(f_i^k)^2.\\
		\end{aligned}
	\end{equation*}
	Then substitute this equation into \eqref{theoremb52},
	\begin{equation*}
		\begin{aligned}
			&\frac{\Delta t(t_n^{-\alpha}-\e_0)}{4\Gamma(1-\alpha)}\sum_{k=1}^n||u^k||^2+\Delta t\sum_{k=1}^n||\delta_xu^k||^2\\
			\leq&\frac{\Delta t^{1-\alpha}(1-\alpha)^{-1}+t_{n-1}\Delta t^{-\alpha}}{2\Gamma(1-\alpha)}||u^0||^2\\
			&\quad+\frac{\Delta t^{1-\frac \alpha 2}\left(1-\frac \alpha 2\right)^{-1}+t_{n-1}\Delta t^{-\frac \alpha 2}}{2\Gamma\left(1-\frac \alpha 2\right)}\left[(u_0^0)^2+\left(u_{N_S}^k\right)^2\right]\\
			&\quad+\frac{\Delta t\Gamma(1-\frac \alpha 2)}{8 \left(t_n^{-\frac \alpha 2}-\e_0\right)}\sum_{k=1}^{n}\left[(hf_0^k)^2+\left(hf_{N_S}^k\right)^2\right]+\frac{\Delta t\Gamma(1-\alpha)}{t_n^{-\alpha}-\e_0}\sum_{k=1}^{n}h\sum_{i=1}^{N_S-1}(f_i^k)^2,
		\end{aligned}
	\end{equation*}
	which is 
	\begin{equation}\label{theoremb53}
		\begin{aligned}
			&\frac{\tilde{\mu}}{4}\Delta t\sum_{k=1}^n||u^k||^2+\Delta t\sum_{k=1}^n||\delta_xu^k||^2\leq \tilde{\rho} ||u^0||^2+\tilde{\kappa}\left[(u_0^0)^2+\left(u_{N_S}^k\right)^2\right]\\
			&\qquad+\frac{\Delta t}{8\tilde{\nu}}\sum_{k=1}^{n}\left[(hf_0^k)^2+\left(hf_{N_S}^k\right)^2\right]+\frac{\Delta t}{\tilde{\mu}}\sum_{k=1}^{n}h\sum_{i=1}^{N_S-1}(f_i^k)^2.
		\end{aligned}
	\end{equation}
	Again we use Lemma \ref{mesh}, Taking $\theta >0$ such that $\frac{1/\theta+1/L}{\theta}=\frac{\tilde{\mu}}{4}$ \\(i.e. $\theta=2\left(1+\sqrt{1+L^2\tilde{\mu}}\right)/(L\tilde{\mu})$), and following
	from Lemma \ref{mesh}, we have 
	\begin{equation}\label{theoremb54}
		\Delta t\sum_{k=1}^n||u^k||_{\infty}^2\leq \frac{2\left(1+\sqrt{1+L^2\tilde{\mu}}\right)}{L\tilde{\mu}}\left(\frac{\tilde{\mu}}{4}\Delta t\sum_{k=1}^{n}||u^k||^2+\Delta t\sum_{k=1}^{n}||\delta_x u^k||^2\right).
	\end{equation}
	Combining \eqref{theoremb54} and \eqref{theoremb53}, we obtain the theorem. 
\end{proof}

\begin{theorem}[Error Analysis]\label{FIDRerror2}
	Suppose $u(x,t)\in C_{x,t}^{4,2}([x_l,x_r]\times [0,T]) $ and $\{u_i^k|0\leq i\leq N_S,\ 0\leq k\leq N_T\}$ are solutions of the problem and the
	difference scheme \eqref{dif1}, respectively. Let $e_i^k=u_i^k-u(x_i,t_k)$. Then there
	exists a positive constant $\tilde{c}_2$ such that
	\begin{equation}\label{FIDRglobalerror}
		\begin{aligned}
			\e_{\text{global}}:=\sqrt{\Delta t \sum_{k=1}^n||e^k||_{\infty}^2}\leq\tilde{c}_2(\Delta t^{2-\alpha}+h^2+\e_0),
		\end{aligned}
	\end{equation}
	where
	\begin{equation*}
		\tilde{c}_2^2=\frac{4\tilde{c}_1^2T\left(1+\sqrt{1+L^2\tilde{\mu}}\right)}{L\tilde{\mu}}\left(\frac{1}{\tilde{\nu}}+\frac{L}{\tilde{\mu}}\right).
	\end{equation*}
\end{theorem}

\begin{proof}
	We observe that the error $e^k$ satisfies the following scheme:
	\begin{equation}\label{dif4}
		\begin{aligned}
			\mathbb{D}_t^\alpha e_i^k&=\delta^2_x e_i^k+T_i^k, &1\leq i\leq N_S-1,1\leq k\leq N_T,\\
			\mathbb{D}_t^\alpha e_0^k&=\frac 2h\left[\delta_x e_{\frac 12}^k- \mathbb{D}_t^{\frac \alpha 2} e_0^k\right]+T_0^k,\\
			\mathbb{D}_t^\alpha e_{N_S}^k&=\frac 2 h \left[-\delta_x e_{N_S-\frac 12}^k-  \mathbb{D}_t^{\frac \alpha 2} e_{N_S}^k\right]+T_{N_S}^k,\\
			e_i^0&=0,&0\leq i\leq N_S,
		\end{aligned}
	\end{equation}
	where the truncation terms $\tilde{T}^k$ at the interior and boundary points are given by the
	formulas
	\begin{equation*}
		\begin{aligned}
			T_i^k=&-\left[ ^C_0D_t^\alpha u(x_i,t_k)- \mathbb{D}_t^\alpha U_i^k\right]+\left[u_{xx}(x_i,t_k)-\delta_x^2U_i^k\right], \quad 1\leq i\leq N_S-1,1\leq k\leq N_T,\\
			T_0^k=&\left\lbrace u_{xx}(x_0,t_k)-\frac 2 h\left[\delta_xU_{\frac 12}^k-u_x(x_0,t_k)\right]-\frac 2 h \left[ ^C_0D_t^{\frac \alpha 2}u(x_0,t_k)- \mathbb{D}_t^{\frac \alpha 2}U_0^k\right] \right\rbrace\\
			&\quad -\left[ ^C_0D_t^\alpha u(x_0,t_k)- \mathbb{D}_t^\alpha U_0^k\right]\\
			T_{N_S}^k=&\left\lbrace u_{xx}(x_{N_S},t_k)+\frac 2 h\left[\delta_xU_{N_S-\frac 12}^k-u_x(x_{N_S},t_k)\right]-\frac 2 h \left[ ^C_0D_t^{\frac \alpha 2}u(x_{N_S},t_k)- \mathbb{D}_t^{\frac \alpha 2}U_{N_S}^k\right] \right\rbrace\\
			&\quad -\left[ ^C_0D_t^\alpha u(x_{N_S},t_k)- \mathbb{D}_t^\alpha U_{N_S}^k\right].\\
		\end{aligned}
	\end{equation*}
	We can show that the truncation terms $T^k$ satisfy the following error bounds
	\begin{equation}\label{theoremb61}
		\begin{aligned}
			|T_i^k|&\leq \tilde{c}_1\left(\Delta t^{2-\alpha}+h^2+\e_0\right),\\
			|T_0^k|&\leq \tilde{c}_1\left(\Delta t^{2-\alpha}+h+\frac{\Delta t^{2-\alpha/2}}{h}+\frac {\e_0} h\right),\\
			|T_{N_S}^k|&\leq \tilde{c}_1\left(\Delta t^{2-\alpha}+h+\frac{\Delta t^{2-\alpha/2}}{h}+\frac {\e_0} h\right).
		\end{aligned}
	\end{equation}
	with $\tilde{c}_1$ some positive constant. Thus, for $h\leq 1$ and $\Delta t\leq 1$, we have 
	\begin{equation}\label{FIDR-c}
		\begin{aligned}
			&\frac{1}{4\tilde{\nu}}\left[(hT_0^k)^2+(hT_{N_S}^k)^2\right]+\frac{2}{\tilde{\mu}} h\sum_{i=1}^{N_S-1}(T_i^k)^2\\
			\leq&\frac{\tilde{c}_1^2}{2\tilde{\nu}}\left(h\Delta t^{2-\alpha}+h^2+\Delta t^{2-\alpha/2}+\e_0\right)^2+\frac{2\tilde{c}_1^2L}{\mu}\left(\Delta t^{2-\alpha}+h^2+\e_0\right)^2\\
			\leq&\frac{2\tilde{c}_1^2}{\tilde{\nu}}\left(\Delta t^{2-\alpha}+h^2+\e_0\right)^2+\frac{2\tilde{c}_1^2L}{\tilde{\mu}}\left(\Delta t^{2-\alpha}+h^2+\e_0\right)^2\\
			\leq&4\tilde{c}_1^2\left(\frac{1}{\tilde{\nu}}+\frac{L}{\tilde{\mu}}\right)(\Delta t^{2-\alpha}+h^2+\e_0)^2.
		\end{aligned}
	\end{equation}
	Applying Theorem \ref{FIDRprior} in \eqref{dif4},
	\begin{equation}
		\begin{aligned}
			\Delta t\sum_{k=1}^{n}||e^k||_{\infty}^2\leq&\frac{\Delta t\left(1+\sqrt{1+L^2\tilde{\mu}}\right)}{L\tilde{\mu}}\sum_{k=1}^{n}\left(\frac{1}{4\tilde{\nu}}[(hT_0^k)^2+(hT_{N_S}^k)^2]+\frac{2}{\tilde{\mu}}h\sum_{i=1}^{N_S-1}(T_i^k)^2\right)\\
			\leq&\frac{4\tilde{c}_1^2T\left(1+\sqrt{1+L^2\tilde{\mu}}\right)}{L\tilde{\mu}}\left(\frac{1}{\tilde{\nu}}+\frac{L}{\tilde{\mu}}\right)\left(\Delta t^{2-\alpha}+h^2+\e_0\right)^2,
		\end{aligned}
	\end{equation}
	and the theorem is proved.
\end{proof}

\section{Comparation and numerical results}\label{Sec-num}
This section presents numerical experiments. In Section 5.1, we compare the approximation accuracy of the sum-of-exponentials in FIR and FIDR, which has been analysed in Section 3. Then we carry out the numerical experiments to test the convergence rates of global error $\e_{\text{global}}$ of FIR \eqref{theorem2} and FIDR \eqref{FIDRglobalerror} in Section 5.2. {Besides, we numerically investigate the trade-off between the number of modes $N_{\text{exp}},N_{A}$ and the global error in the small $\alpha$ regime.} All data generated or analysed during this study are included in this published article.

\subsection{Comparison of the errors in sum-of-exponentials} \label{num:sum}

From Theorem \ref{FIRerror2} and Theorem \ref{FIDRerror2}, we learn that the global error of FIR $\e_{\text{FIR}}$ and the global error of FIDR $\e_{\text{FIDR}}$ satisfy, 

\begin{equation}\label{Compare}
	\begin{aligned}
		\e_{\text{FIR}}\leq c_2(\Delta t^{2-\alpha}+h^2+{\alpha}\e),\qquad 
		\e_{\text{FIDR}}\leq\tilde{c}_2(\Delta t^{2-\alpha}+h^2+\e_0),
	\end{aligned}
\end{equation}
where
\begin{equation*}
	\begin{aligned}
		&c_2^2=\frac{4c_1^2T\left(1+\sqrt{1+L^2\mu}\right)}{L\mu}\left(\frac 1\nu+\frac{L}{\mu}\right), 
		&\tilde{c}_2^2=\frac{4\tilde{c}_1^2T\left(1+\sqrt{1+L^2\tilde{\mu}}\right)}{L\tilde{\mu}}\left(\frac{1}{\tilde{\nu}}+\frac{L}{\tilde{\mu}}\right),	
	\end{aligned}
\end{equation*}
and

\begin{equation*}
	\begin{aligned}
		&\mu=\frac{t_n^{-\alpha}-2\alpha \e t_{n-1}}{\Gamma(1-\alpha)},\quad &\tilde{\mu}=\frac{t_n^{-\alpha}-\e_0}{\Gamma(1-\alpha)},\quad
		&\nu=\frac{t_n^{-\frac \alpha 2}-\alpha \e t_{n-1}}{\Gamma(1-\frac \alpha 2)},\quad\\ &\tilde{\nu}=\frac{t_n^{-\frac \alpha 2}-\e_0}{\Gamma(1-\frac \alpha 2)},\quad 
		&c_1/\tilde{c}_1=O(1).
	\end{aligned}
\end{equation*}

For $T=O(1)$ and $\Delta t$ and $h$ are sufficiently small, the errors are dominated by $\alpha\e $ and $\eps_0$ respectively.

In FIR and FIDR, we set $N_{\text{exp}}=25$ ($a=3,\ b=10,\ n_1=4,\ n_2=3$). Meanwhile, we set $\alpha=0.1$, and in the following figure we can find that the global error of these two scheme are close. Here we compute the error of the sum-of-exponentials approximation.\\
For FIR and FIDR, define the error of sum-of-exponential approximations respectively,
\begin{equation}\label{errorFIRFIDR}
	\e(t)=\left| \frac{1}{t^{1+\alpha}}-\sum_{i=1}^{N_{A}}w_ie^{-s_it}\right|,\qquad \e_0(t)=\left| \frac{1}{t^{\alpha}}-\sum_{i=1}^{N_{A}}\tilde{w}_ie^{-\tilde{s}_it}\right|.
\end{equation}
Then $\e$ and $\e_0$ in \eqref{Compare} can be written as
\begin{equation}
	\begin{aligned}
		\e=\max_{0\leq t\leq T}\e(t),\qquad 
		\e_0=\max_{0\leq t\leq T}\e_0(t).
	\end{aligned}
\end{equation}


\begin{figure}[H]
	\centering
	\includegraphics[width=0.7\textwidth]{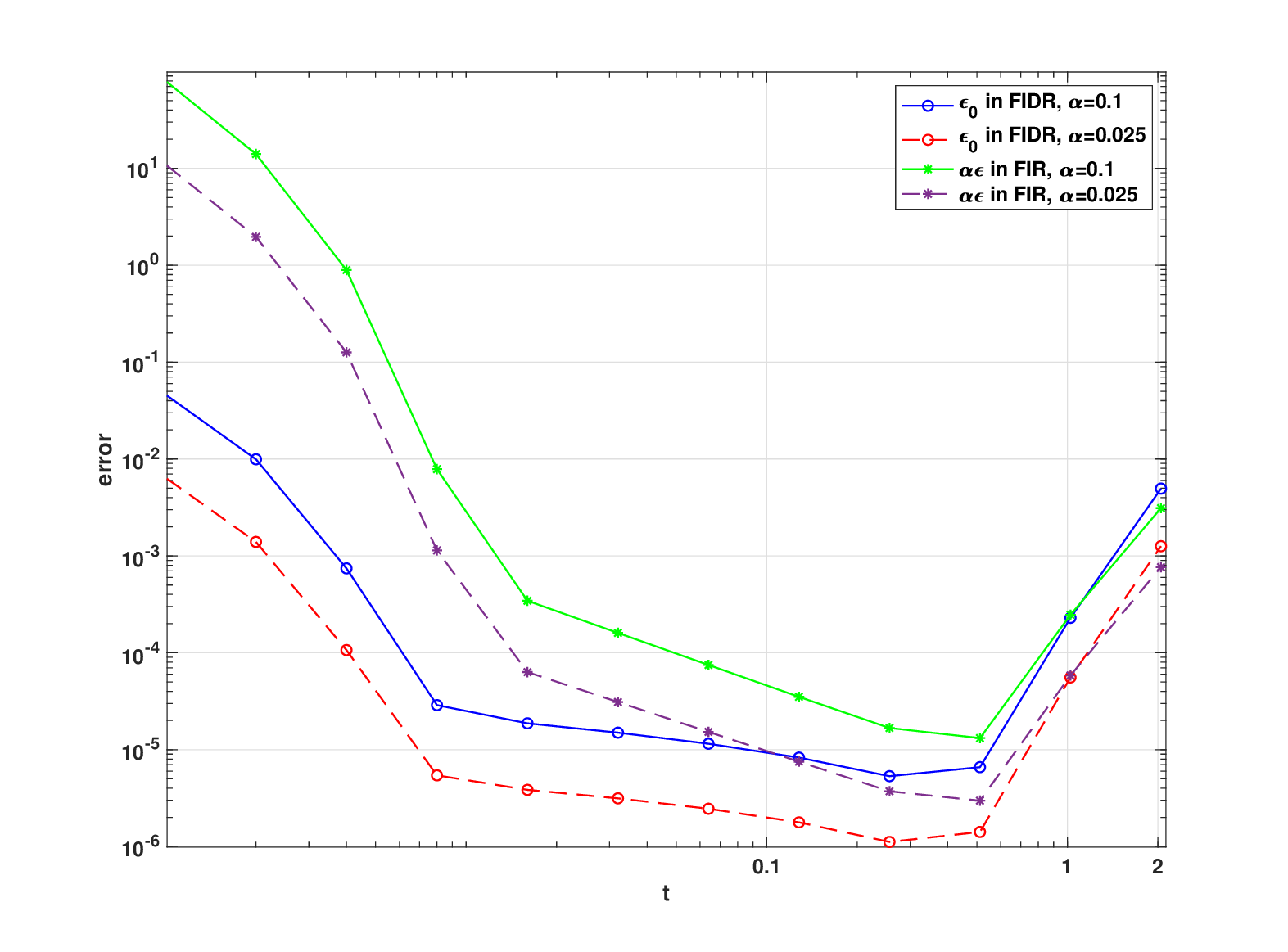}  
	\caption{Error of sum-of-exponentials approximation. The total number of nodes $N_{\text{exp}}=25$ is fixed. X-label denotes the time and Y-label denotes the error. The errors of FIR and FIDR are $\alpha\e(t)$ and $\e_0(t)$, respectively, where $\e(t)$ and $\e_0(t)$  are defined in \eqref{errorFIRFIDR}.}
\end{figure}

This figure shows that when $t<1$, the error of sum-of-exponentials approximation works in FIDR is smaller than in FIR. However, when $t$ becomes larger, the error in FIR and in FIDR are similar. When $T=O(1)$, the error bound of diffusive approximation is smaller than the error bound of fast evaluation. This result is also supported by the numerical experiments in the following section. 

\subsection{Numerical experiments}
In this part first we set an linear fractional diffusion equation \eqref{test1}, \eqref{test2} as the test function, which has been used as the test function in \cite{JiangZhang}. Problem \eqref{test1} has exact solution \eqref{exact}, so we can calculate the global error and compare the efficiency between FIR and FIDR. We also test an nonlinear equation \eqref{nonlinear}, but there is no exact solution this time. We instead present the convergence rate of  the related error  of FIDR.

We first consider the following initial value problem of the linear fractional diffusion equation \eqref{problem}, which is presented in Section 3.3 in \cite{JiangZhang}. 
\begin{equation}\label{test1}
	\begin{aligned}
		^C_0D_t^\alpha u(x,t)&=u_{xx}(x,t)+f(x,t), &x\in\Omega,\ t>0,\\
		u(x,0)&=x_0(x),&x\in \Omega,\\
		\frac{\partial u(x,t)}{\partial x}&=\frac{1}{\Gamma (1-\frac \alpha 2)}\int_0^t \frac{u_s(x,s)}{(t-s)^{\frac \alpha 2}}ds:=\ ^C_0D_t^{\frac \alpha 2} u(x,t), &x=x_l\\
		\frac{\partial u(x,t)}{\partial x}&=-\frac{1}{\Gamma (1-\frac \alpha 2)}\int_0^t \frac{u_s(x,s)}{(t-s)^{\frac \alpha 2}}ds:=- ^C_0D_t^{\frac \alpha 2} u(x,t), &x=x_r.\\
	\end{aligned}
\end{equation}
Take the computational domain $\Omega=[0,\pi]$, and set
\begin{equation}\label{test2}
	\begin{aligned}
		f(x,t)=&\Gamma(4+\alpha)x^4(\pi-x)^4\exp(-x)t^3/6-x^2(\pi-x)^2	\left\{t^{3+\alpha}\exp(-x)\right.\\
		&\quad \left[x^2(56-16x+x^2)-2\pi x(28-12x+x^2)+\pi^2(12-8x+x^2)\right]\\
		&\quad \left. +4(3\pi^2-14\pi x+14x^2) \right\},\\
		u_0(x)=&\left\{\begin{matrix}
			x^4(\pi-x)^4, &x\in\Omega,\\
			0,&x\notin \Omega.
		\end{matrix}\right.
	\end{aligned}
\end{equation}
This problem has the exact solution given by the formula
\begin{equation}\label{exact}
	u(x,t)=x^4(\pi-x)^4\left[\exp(-x)t^{3+\alpha}+1\right], \quad (x,t)\in \Omega\times(0,T].
\end{equation}

{Here we present the numerical results with different $\alpha$. The related error shown in figures below is defined by
	\begin{equation}\label{relatederror}
		\e_{\text{related}}:=\frac{\e_{\text{global}}}{\sqrt{\Delta t\sum_{k=1}^n||u^k||_\infty^2}}.
	\end{equation}
	where $u$ is the real solution in \eqref{exact}, $u^k=(u(x_1,t_k),u(x_2,t_k),\cdots,u(x_{N_S},t_k))$, $\e_{\text{global}}$ is the global error defined in \eqref{FIDRglobalerror}}
\begin{figure}[H]
	\centering  
	\includegraphics[width=0.7\textwidth]{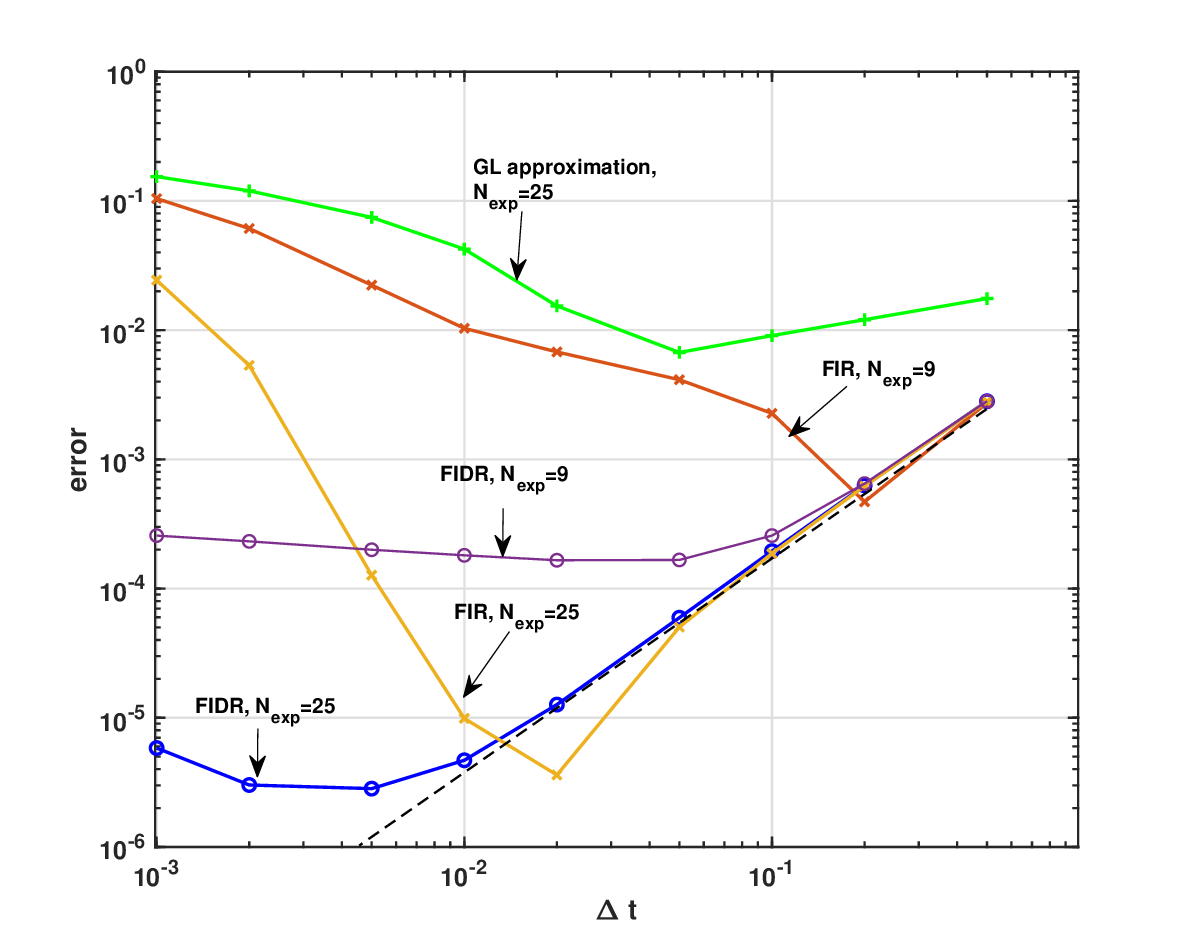}  
	\caption{The numerical experiment of problem \eqref{test1} when $\alpha=0.1$. The plot of related error (defined in \eqref{relatederror}) versus $\Delta t$, $h=10^{-3}$, $T=1$ is fixed. $N_{\text{exp}}$ is the number of nodes in sum-of-exponentials. We use GL approximation as a comparative scheme. We give the dotted line (black line) to show that FIDR converges in linear rate when $t$ is not too small (which is same in Figure \ref{Fig-lin05}, \ref{Fig-lin07} below).}  
	\label{Fig-lin01}   
\end{figure}

\begin{figure}[H]
	\centering  
	\includegraphics[width=0.7\textwidth]{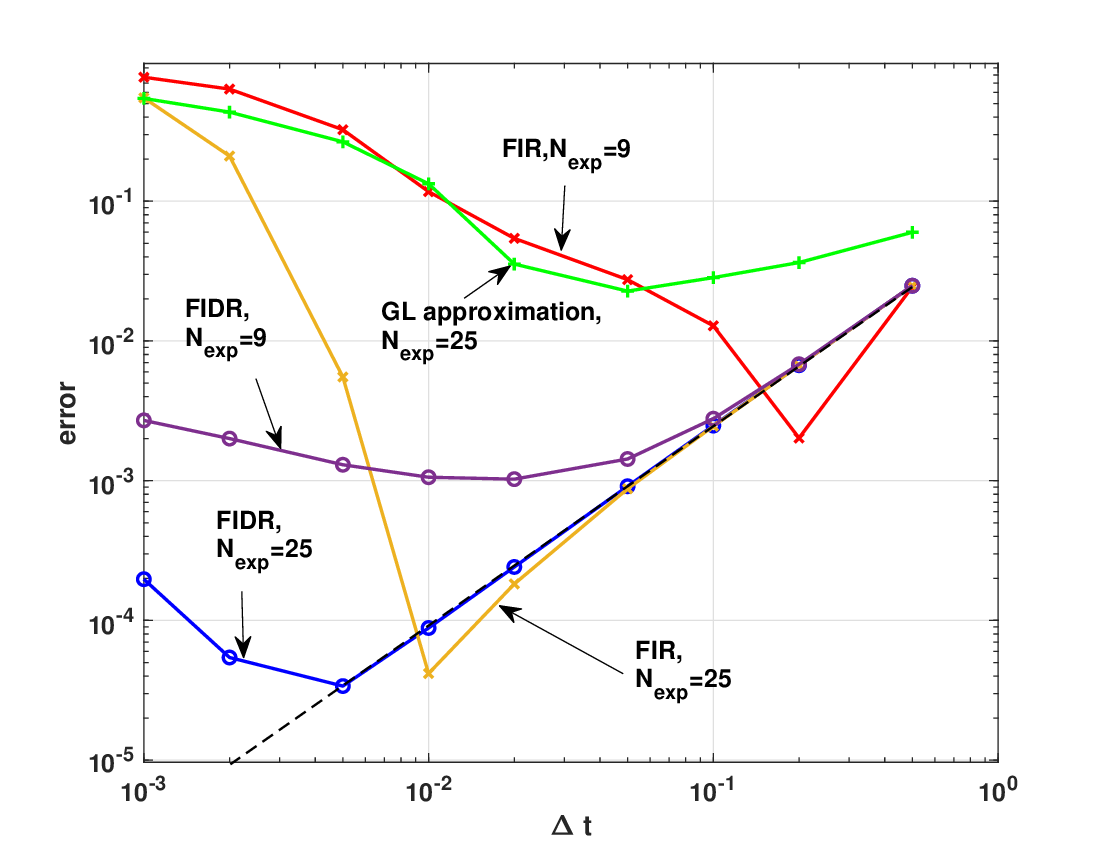}  
	\caption{The numerical experiment of problem \eqref{test1} when $\alpha=0.5$. The plot of related error (defined in \eqref{relatederror}) versus $\Delta t$, $h=10^{-3}$, $T=1$ is fixed. $N_{\text{exp}}$ is the number of nodes in sum-of-exponentials.}  
	\label{Fig-lin05}   
\end{figure}
\begin{figure}[H]
	\centering  
	\includegraphics[width=0.7\textwidth]{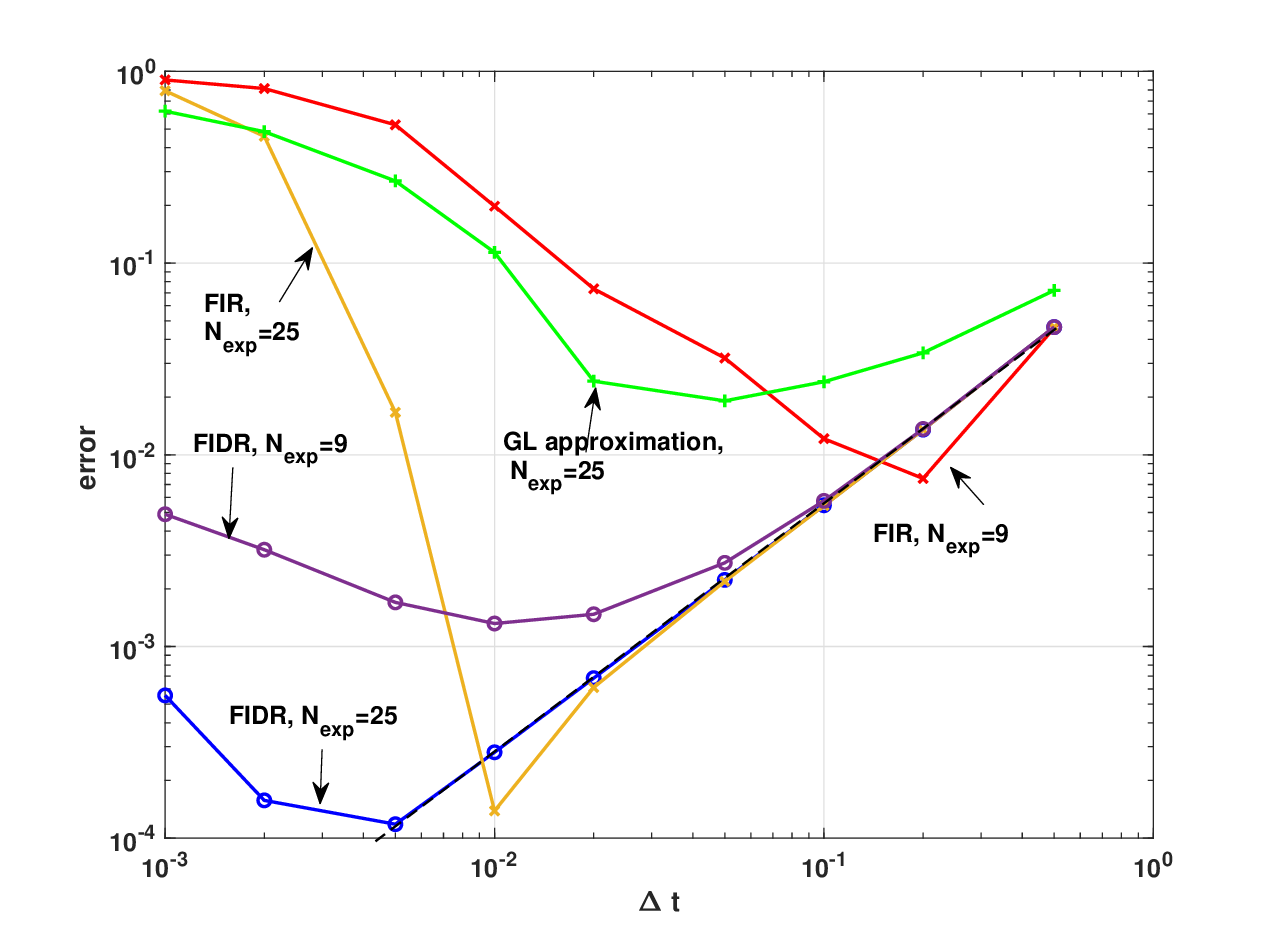}  
	\caption{The numerical experiment of problem \eqref{test1} when $\alpha=0.7$. The plot of related error (defined in \eqref{relatederror}) versus $\Delta t$, $h=10^{-3}$, $T=1$ is fixed. $N_{\text{exp}}$ is the number of nodes in sum-of-exponentials.}  
	\label{Fig-lin07}
\end{figure}
Figure \ref{Fig-lin01}, \ref{Fig-lin05}, \ref{Fig-lin07} show that the global error of FIDR is smaller than FIR and GL approximation when $N_{\text{exp}}=25$ (or $N_{\text{exp}}=9$) is same for the three schemes when  $\alpha=0.1,\ 0.5,\ 0.7$. Recall that $N_{\text{exp}}$ is the storage cost of the scheme. For the same scheme, the global error decreases when $N_{\text{exp}}$ grows. Furthermore, when $\Delta t$ isn't too small, {in these experiments which means $\Delta t> 10^{-2}$ when $N_{\text{exp}}=25$ and $\Delta t>10^{-1}$ when $N_{\text{exp}}=9$}, FIDR converges in linear rate (compare with the dotted line). {Here we set Gr\"unwald-Letnikov and central-difference approximations (also called GL approximation) as a comparison scheme. GL approximation is widely used in engineering \cite{CCMH}.} It is absolute stable but need more storage. We give the proof of the stability of GL approximation in the appendix.

{Remark that when $\alpha=0.7$, $Q\approx-0.73$ which is defined by \eqref{defQ} in viscoelastic models. This is a challenging case for practical simulation in engineering  with such a small $Q$, but FIDR can still compute it well with little storage cost.   }

The table below shows the related error and simulation time of FIDR and FIR.
\begin{table}[H]
	\centering
	\begin{tabular}{c|cc|cc}
		\hline
		$\Delta t$& time of FIDR & time of FIR & error of FIDR& error of FIR	\\
		\hline
		1e$-$01 &3.76e$-$01 &3.43e$-$01&1.94e$-$04&1.86e$-$04  \\
		\hline
		5e$-$02 &5.26e$-$01&5.73e$-$01 &5.94e$-$05&5.03e$-$05 \\
		\hline
		1e$-$02&2.32e$+$00 &2.39e$+$00&4.68e$-$06&9.89e$-$06  \\
		\hline
		5e$-$03 &4.60e$+$00 &4.59e$+$00&2.82e$-$06& 1.26e$-$04 \\
		\hline
		1e$-$03 &2.33e$+$01 &2.32e$+$01  &5.83e$-$06&2.43e$-$02 \\
		\hline
	\end{tabular}
	\caption{The related error and simulation time of FIDR and FIR when solving differential equation \eqref{test1}, \eqref{test2}. Where $\alpha=0.1,\ h=10^{-3},\ N_{\text{exp}}=25$, $T=1$. The related error is defined in \eqref{relatederror}}
	\label{Tab-Linear}
\end{table}
Table \ref{Tab-Linear} shows that FIR and FIDR cost similar simulation time when $N_{\text{exp}},\ h,\ \Delta t$ are the same. But FIDR achieves smaller global error. From Table 2 we can see that when $\alpha=0.1$ is small, both of FIR and FIDR can get accurate result with small storage cost. Furthermore, here $T=O(1)$, and the error of FIDR is smaller than the error of FIR, especially when $\Delta t<10^{-2}$.

We also calculate nonlinear diffusion PDE.
\begin{equation}\label{nonlinear}
	\begin{aligned}
		f(u)&=-u(1-u),\\
		u(x,0)&=\exp(-10(x-0.5)^2)+\exp(-10(x+0.5)^2).
	\end{aligned}
\end{equation} 
However, there isn't exact solution of this nonlinear PDE, so it is hard to compare the result of different schemes. We use this problem to test the self convergence of FIDR. the reference solution is computed over very small mesh sizes $h=\Delta t=10^{-4}$.

\begin{figure}[H]
	\centering  
	\includegraphics[width=0.7\textwidth]{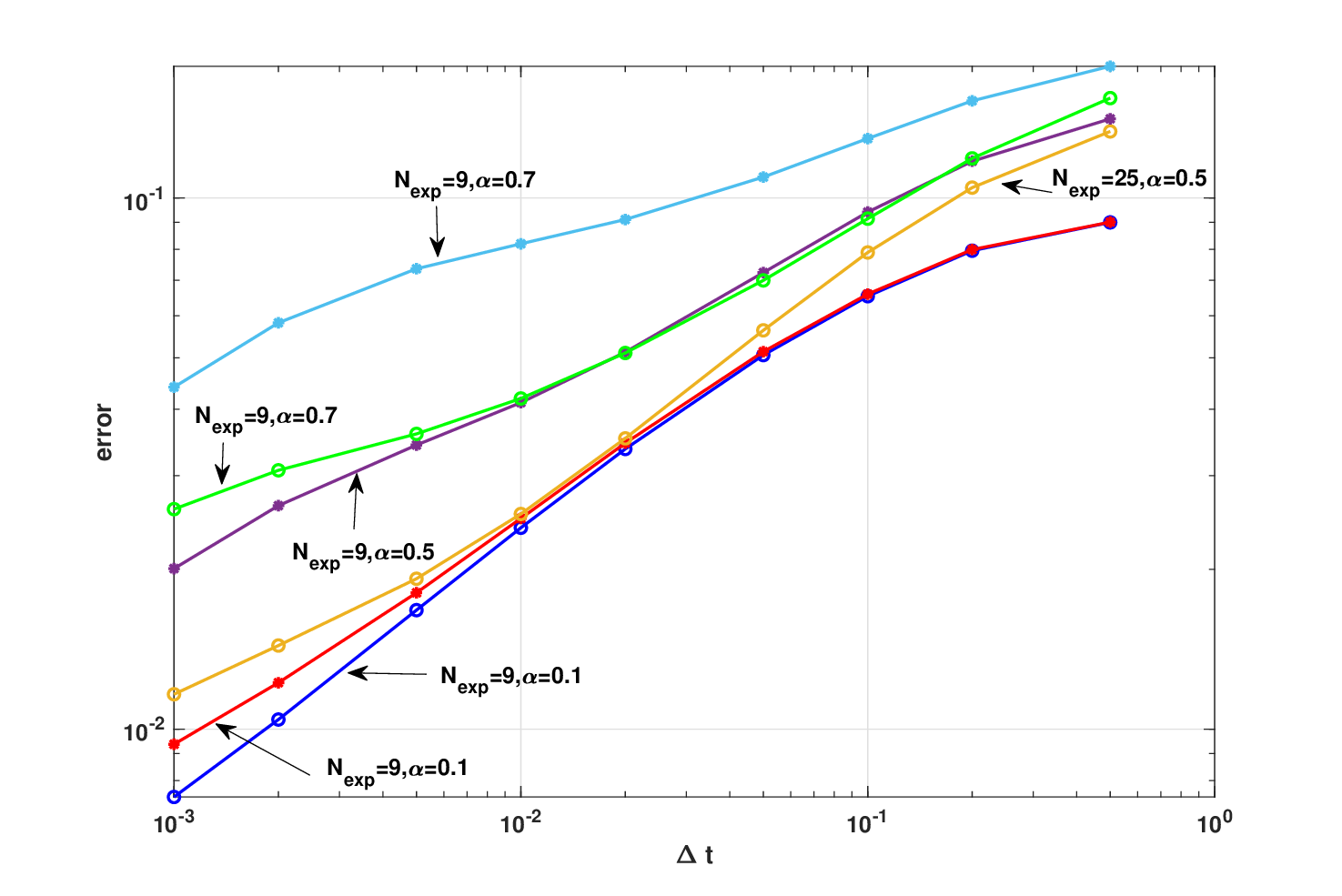}  
	\caption{The plot of related error (defined in \eqref{relatederror}) when solving \eqref{nonlinear} versus $\Delta t$. $h=10^{-3}$ is fixed. Different lines represent results with different $\alpha$ and $N_{\text{exp}}$.}  
	\label{Fig-nonlin}
\end{figure}
{Figure \ref{Fig-nonlin} shows FIDR still converges in linear rate in nonlinear differential equation. }

In the end of this section, we remark that in all the numerical experiments we have done, we only use few nodes in the sum-of-exponentials approximation, which make the global error of FIR blow up when $\Delta t$ is small. However, if the number of nodes $N_{\text{exp}}$ is large enough, the global error of FIR will stay small even when $\Delta t$ is small. To show this, we list the error of FIR with different $N_{\text{exp}}$ in the following table.
\begin{table}[H]
	\centering
	\begin{tabular}{c|cc|cc}
		\hline
		&FIDR ($N_{\text{exp}}=25$) & FIDR ($N_{\text{exp}}=40$) & FIR ($N_{\text{exp}}=25$) & FIR ($N_{\text{exp}}=40$)	\\
		\hline
		$\alpha=0.1$&5.83e$-$06&2.39e$-$06 &2.43e$-$02&2.04e$-$05 \\
		\hline
		$\alpha=0.5$ &1.97e$-$04&5.23e$-$06 &5.49e$-$01 &3.13e$-$04\\
		\hline
		$\alpha=0.7$& 5.56e$-$04 & 1.63e$-$05&7.93e$-$01&6.91e$-$04  \\
		\hline
	\end{tabular}
	\caption{ the related error for different $\alpha$ of FIDR and FIR, where $\Delta t=10^{-3}$ is fixed.}
	\label{Tab-comp}
\end{table}
{Table \ref{Tab-comp} shows that when $N_{\text{exp}}$ is large enough (here $N_{\text{exp}}\geq 40$), the blowing up of the error of FIR does not appear, which is the case  considered in \cite{JiangZhang}. However, the global error of FIR is still larger than the global error of FIDR when $N_{\text{exp}}$ is the same.}

\section{Conclusion}\label{Sec-conc}
In this paper, the relationship between the global error of FIR \cite{JiangZhang}, FIDR \cite{shen2018fast}  and $\alpha$ in the Caputo derivative is studied for the first time. We prove that the global error of FIR and FIDR is reduced when $\alpha$ gets smaller, thus we can reduce the modes considerably in sum-of-exponential approximation when $\alpha$ is small, then the storage and the computational cost will decrease. This result successfully fits the discovery in engineering work \cite{ZhanZhuangSun2017, ZhanZhuangLiu, ZZFL}.

Furthermore, we discover that the global error of FIDR is smaller than the global error of FIR  in all of the numerical experiments when  the number of modes $N_{\text{exp}},N_A$ are the same, {especially when $\alpha$ is small}. {The stability of FIDR is proved and we give a more accurate upper bound of the global error of FIDR compared with the analysis in \cite{shen2018fast}.} The numerical experiments about FIR and FIDR completely confirming the theoretical results carried out in this paper. 

{In future work, we will focus on the high order scheme based on  scheme FIDR.} Current FIDR is a first order scheme, since we approximate the derivative $u'(\tau)$ with first order method. We plan to approximate $u'(\tau)$ with higher-order scheme for Caputo fractional derivative. It is also meaningful to explore more advanced technique to efficiently calculate the Caputo derivatives with large $\alpha$, thus leading to a system with strong attenuation and dispersion. 

\section*{Acknowledgement}
Q. Zhan is supported by the NSFC grant No. 92066105  and CEMEE Grant 2021K0302A.\\ \\
Z. Zhou is supported by the National Key R\&D Program of China, Project Number 2020YFA0712000 and NSFC grant No. 11801016, No. 12031013. Z. Zhou is also partially supported by Beijing Academy of Artificial Intelligence (BAAI).

\newpage
\begin{appendices}
	\section{The Proof of Lemma \ref{app}}
	\pmb{Lemma \ref{app}} For any mesh functions $g=\{g^k|0\leq k\leq N \}$ defined on $\Omega_t$, the
	following inequality holds:
	\begin{equation}\label{appendix1}
		\Delta t\sum_{k=1}^n( \mathbb{D}_t^\alpha g^k)g^k\geq \frac{t_n^{-\alpha}-2\alpha\e t_{n-1}}{2\Gamma(1-\alpha)}\Delta t\sum_{k=1}^n(g^k)^2-\frac{t_n^{1-\alpha}-\alpha(1-\alpha)\e t_{n-1}\Delta t}{\Gamma(2-\alpha)}(g^0)^2.
	\end{equation}
	Recall the definition of $\mathbb{D}_t^\alpha$ is
	\begin{equation}\label{appendix2}
		\begin{aligned}
			\mathbb{D}_t^\alpha u^n=&\frac{\Delta t^{-\alpha}}{\Gamma(1-\alpha)}\left(\frac{u^n}{1-\alpha}-(\frac{\alpha}{1-\alpha}+a_0)u^{n-1}\right.\\
			&\quad\left.-\sum_{l=1}^{n-2}(a_{n-l-1}+b_{n-l-2})u^l-(b_{n-2}+\frac{1}{n^\alpha})u^0\right),
		\end{aligned}
	\end{equation}
	where
	\begin{equation*}
		\begin{aligned}
			&a_n=\alpha\Delta t^\alpha\sum_{j=1}^{N_{\text{exp}}}\omega_j e^{-ns_j\Delta t}\lambda_j^1, &b_n=\alpha \Delta t^\alpha\sum_{j=1}^{N_{\text{exp}}}\omega_j e^{-ns_j\Delta t}\lambda_j^2,\\
			&\lambda_j^1=\frac{e^{-s_j\Delta t}}{s_j^2\Delta t}(e^{-s_j\Delta t}-1+s_j\Delta t),
			&\lambda_j^2=\frac{e^{-s_j\Delta t}}{s_j^2\Delta t}(1-e^{-s_j\Delta t}-e^{-s_j\Delta t}s_j\Delta t).
		\end{aligned}
	\end{equation*}\\
	
	proof: Applying the definition \eqref{appendix2} of the fast evaluation scheme and the
	Cauchy-Schwarz inequality, we have
	\begin{equation}\label{appendix3}
		\begin{aligned}
			( \mathbb{D}_t^\alpha g^k)g^k=&\frac{\Delta t^{-\alpha}}{\Gamma(1-\alpha)}\Bigg(\frac{(g^k)^2}{1-\alpha}-\left(\frac{\alpha}{1-\alpha}+a_0\right)g^{k-1}g^k\\
			&\quad-\sum_{l=1}^{k-2}(a_{k-l-1}+b_{k-l-2})g^lg^k-\left(b_{k-2}+\frac{1}{k^\alpha}\right)g^0g^k\Bigg)\\
			\geq&\frac{1}{\Delta t^\alpha\Gamma(1-\alpha)}\left[\left(\frac{2-\alpha}{2(1-\alpha)}-\frac 12 \sum_{l=0}^{k-2}(a_l+b_l)-\frac{1}{2k^\alpha}\right)(g^k)^2\right.\\
			&\quad-\frac 12 \left(\frac{\alpha}{1-\alpha}+a_0\right)(g^{k-1})^2-\frac 12 \sum_{l=1}^{k-2}(a_{k-l-1}+b_{k-l-2})(g^l)^2\\
			&\quad\left. -\frac 12 \left(b_{k-2}+\frac{1}{k^{\alpha}}\right)(g^0)^2 \right].
		\end{aligned}
	\end{equation}
	Summing the above inequality from k=1 to n, we obtain
	\begin{equation}\label{appendix4}
		\begin{aligned}
			\Delta t\sum_{k=1}^n( \mathbb{D}_t^\alpha g^k)g^k\geq& \frac{\Delta t^{1-\alpha}}{\Gamma(1-\alpha)}\Bigg[\sum_{k=2}^n\left(\frac{2-\alpha}{2(1-\alpha)}-\frac 12 \sum_{l=0}^{k-2}(a_l+b_l)-\frac{1}{2k^\alpha}\right)(g^k)^2\\
			&\quad-\frac 12\sum_{k=2}^n (\frac{\alpha}{1-\alpha}+a_0)(g^{k-1})^2-\frac 12 \sum_{k=2}^n\sum_{l=1}^{k-2}(a_{k-l-1}+b_{k-l-2})(g^l)^2\\
			&\quad -\frac 12\sum_{k=2}^n \left(b_{k-2}+\frac{1}{k^{\alpha}}\right)(g^0)^2+\frac{1}{2(1-\alpha)}(g^1)^2-\frac{1}{2(1-\alpha)}(g^0)^2 \Bigg]\\
			=&\frac{\Delta t^{1-\alpha}}{\Gamma(1-\alpha)}\left(\sum_{k=1}^n\left(C_k(g^k)^2\right)-C_0(g^0)^2\right).
		\end{aligned}
	\end{equation}
	where the coefficients $C_k\ (k=0,1,\cdots,n)$  are given by the formula
	\begin{equation}\label{appendix5}
		\begin{aligned}
			C_k=\left\lbrace\begin{matrix}
				\frac{1}{2(1-\alpha)}+\frac 12\sum_{k=2}^n \left(b_{k-2}+\frac{1}{k^{\alpha}}\right), &k=0,\\
				\frac 12 -\frac 12 \sum_{l=0}^{n-2}(a_l+b_l)+\frac 12 b_{n-2}, &k=1,\\
				1- \frac 12\sum_{l=0}^{k-2}(a_l+b_l)-\frac{1}{2k^\alpha}-\frac 12\sum_{l=0}^{n-k-1}(a_l+b_l)+\frac 12 b_{n-k-1}, &2\leq k\leq n,\\
				\frac{2-\alpha}{2(1-\alpha)}-\frac 12 \sum_{l=0}^{k-2}(a_l+b_l)-\frac{1}{2k^\alpha}, &k=n
			\end{matrix}\right.
		\end{aligned}
	\end{equation}
	we have the estimate
	\begin{equation}\label{appendix6}
		\frac{1}{t^{1+\alpha}}-\e\leq \sum_{j=1}^{N_{\text{exp}}}\omega_je^{-s_j t}\leq \frac{1}{t^{1+\alpha}}+\e.
	\end{equation}
	It is also straightforward to verify that
	\begin{equation}\label{appendix7}
		\sum_{l=0}^{k-2}(a_l+b_l)=\alpha\Delta t^\alpha \int _{\Delta t}^{k\Delta t}\sum_{j=1}^{N_{\text{exp}}}\omega_j e^{-s_j t}dt.
	\end{equation}
	Combining \eqref{appendix6} and \eqref{appendix7}, we obtain
	\begin{equation}\label{appendix8}
		\left(1-\frac{1}{k^\alpha}\right)-\alpha\Delta t^\alpha t_{k-1}\e\leq\sum_{l=0}^{k-2}(a_l+b_l)\leq\left(1-\frac{1}{k^\alpha}\right)+\alpha\Delta t^\alpha t_{k-1}\e. 
	\end{equation}
	Substituting \eqref{appendix8} into \eqref{appendix5} yields the following estimates
	\begin{equation}\label{appendix9}
		\begin{aligned}
			&C_0\leq \frac{n^{1-\alpha}}{(1-\alpha)}+\alpha\Delta t^\alpha t_{n-1}\e,\\
			&C_1\geq \frac 12 -\frac 12 \sum_{l=0}^{n-2}(a_l+b_l)\geq \frac{1}{2n^\alpha}-\alpha \Delta t^\alpha t_{n-1}\e,\\
			&C_k\geq \frac{1}{2n^\alpha}-\alpha \Delta t^\alpha t_{n-1}\e, &2\leq k\leq n-1,\\
			&C_n\geq \frac{2-\alpha}{2(1-\alpha)}-\sum_{l=0}^{n-2}(a_l+b_l)-\frac{1}{2n^\alpha}\geq\frac{1}{2n^\alpha}-\alpha \Delta t^\alpha t_{n-1}\e.
		\end{aligned}
	\end{equation}
	Combining \eqref{appendix9} and \eqref{appendix4}, we obtain the Lemma.

	\section{Stability of GL approximation}
	Recall the GL approximation
	\begin{equation}\label{appendix3.1}
		\begin{aligned}
			^{GL}_0\mathbb{D}_t^p u(t)&= \Delta t^{-p}\sum_{m=0}^{t/\Delta t}(-1)^m\left(\begin{matrix}
				p\\ m
			\end{matrix}\right)u(t-m\Delta t)\\
			&=\Delta t^{-p}\sum_{m=0}^{n}(-1)^m\left(\begin{matrix}
				p\\ m
			\end{matrix}\right)U^{n-m}.
		\end{aligned}
	\end{equation}
	
	Define 
	\begin{equation}\label{fourier}
		U^{(k)m}=\lambda_k^m,
	\end{equation}
	which is the particular solution in fourier wave pattern. Then \eqref{appendix3.1} becomes 
	\begin{equation}
		\begin{aligned}
			\Delta t^{-p}\sum_{m=0}^{n}(-1)^m\left(\begin{matrix}
				p\\ m
			\end{matrix}\right)\lambda_k^{n-m}=0,
		\end{aligned}
	\end{equation}
	thus
	\begin{align*}
		\sum_{m=0}^{n}(-1)^m\left(\begin{matrix}
			p\\ m
		\end{matrix}\right)\lambda_k^{n-m}=0.\\
	\end{align*}
	Our purpose is to find whether there exists a solution $\lambda$ satisfies $|\lambda|>1$.
	If $p$ is a positive integer, when $n\geq p$ we have
	\begin{align*}
		(\lambda_k-1)^p&=0,\\
		\lambda_k&=1,
	\end{align*}
	if $0<p<1$ is not an integer. 
	Using the known property of the binomial coefficients
	\begin{equation}
		\left(\begin{matrix}
			p\\ m
		\end{matrix}\right)=\left(\begin{matrix}
			p-1\\ m
		\end{matrix}\right)+\left(\begin{matrix}
			p-1\\ m-1
		\end{matrix}\right),
	\end{equation}
	\begin{equation}\label{fourier2}
		\begin{aligned}
			0&=\sum_{m=0}^{n}(-1)^m\left(\begin{matrix}
				p\\ m
			\end{matrix}\right)\lambda_k^{n-m}\\
		&=\sum_{m=0}^{n}(-1)^m\left(\begin{matrix}
				p-1\\ m
			\end{matrix}\right)\lambda_k^{n-m}-\sum_{m=0}^{n-1}(-1)^m\left(\begin{matrix}
				p-1\\ m
			\end{matrix}\right)\lambda_k^{n-m-1}\\
			&=\left(\sum_{m=0}^{n-1}(-1)^m\left(\begin{matrix}
				p-1\\ m
			\end{matrix}\right)\lambda_k^{n-m-1} \right)(\lambda_k-1)+(-1)^n\left(\begin{matrix}
				p-1\\ n
			\end{matrix}\right).
		\end{aligned}
	\end{equation}
	Define
	\begin{equation}
		\begin{aligned}
			F_{p,n}(\lambda_k):=\left(\sum_{m=0}^{n-1}(-1)^m\left(\begin{matrix}
				p-1\\ m
			\end{matrix}\right)\lambda_k^{n-m-1} \right)(\lambda_k-1)+(-1)^n\left(\begin{matrix}
				p-1\\ n
			\end{matrix}\right).
		\end{aligned}
	\end{equation}
	\pmb{Absolute stability:} Consider the following ODE model
	\begin{equation}
		\begin{aligned}
			^C_0D_t^p u(t)=c u(t)+f(t).
		\end{aligned}
	\end{equation}
	The backward Euler method for this ODE is 
	\begin{equation}
		\begin{aligned}
			^{GL}_0\mathbb{D}_t^p U^n=cU^n.
		\end{aligned}
	\end{equation}
	where $Re(c)\leq0$.\\ 
	Applying \eqref{appendix3.1} and \eqref{fourier} to consider the absolute stability.
	\begin{equation}\label{euler}
		\begin{aligned}
			\Delta t^{-p}\sum_{m=0}^{n}(-1)^m\left(\begin{matrix}
				p\\ m
			\end{matrix}\right)\lambda_k^{n-m}-c \lambda_k^n=0,\\
			(1-\Delta t^p c) \lambda_k^n+\sum_{m=1}^{n}(-1)^m\left(\begin{matrix}
				p\\ m
			\end{matrix}\right)\lambda_k^{n-m}=0.
		\end{aligned}
	\end{equation}
	Since $p\in(0,1)$, we have $(-1)^m\left(\begin{matrix}
		p\\ m
	\end{matrix}\right)<0$, where $1\leq m\leq n$. Recall that
	\begin{equation*}
		F_{p,n}(\lambda_k)=\sum_{m=0}^{n}(-1)^m\left(\begin{matrix}
			p\\ m
		\end{matrix}\right)\lambda_k^{n-m}=\lambda_k^n+\sum_{m=1}^{n}(-1)^m\left(\begin{matrix}
			p\\ m
		\end{matrix}\right)\lambda_k^{n-m},
	\end{equation*}
	thus,
	\begin{equation}
		\begin{aligned}
			\sum_{m=1}^n\left| (-1)^m\left(\begin{matrix}
				p\\ m
			\end{matrix}\right)\right|&=-\sum_{m=1}^{n}(-1)^m\left(\begin{matrix}
				p\\ m
			\end{matrix}\right)\\
			&=1-F_{p,n}(1)\\
			&=1-(-1)^n\left(\begin{matrix}
				p-1\\ n
			\end{matrix}\right),
		\end{aligned}
	\end{equation}
	where we have used equation \eqref{fourier2}.
	
	Define $z=\Delta t^pc$, then $Re(z)\leq0$, $|1-z|\geq1$. If we assume $\lambda_0$ with $|\lambda_0|\geq1$ is a root to \eqref{euler}, then for $m<n$, we have $|\lambda_0|^m\leq|\lambda_0|^n$. Then we get
	\begin{equation}
		\begin{aligned}
			|1-z||\lambda_0|^n&\leq\sum_{m=1}^{n}\left|(-1)^m\left(\begin{matrix}
				p\\ m
			\end{matrix}\right)\right||\lambda_0|^{n-m}\\
			&\leq\left(\sum_{m=1}^{n}\left|(-1)^m\left(\begin{matrix}
				p\\ m
			\end{matrix}\right)\right|\right)|\lambda_0|^n\\
			&=\left(1-(-1)^n\left(\begin{matrix}
				p-1\\ n
			\end{matrix}\right)\right)|\lambda_0|^n\\
			&<|\lambda_0|^n.
		\end{aligned}
	\end{equation} 
	which is a contradiction. So every solution $\lambda_k$ of \eqref{euler} saftisfies $|\lambda_k|<1$, which mean that GL approximation is A-stable and zero stable.

\end{appendices}

\bibliographystyle{amsrefs}
\bibliography{caputo}

\begin{bibdiv}
\begin{biblist}

\bib{AlkheirIbnkahla}{misc}{
      author={Alkheir, A.A.},
      author={Ibnkahla, M.},
       title={An accurate approximation of the exponential integral function
  using a sum of exponentials},
        date={2013},
        note={IEEE COMMUNICATIONS LETTERS, VOL. 17, NO. 7},
}

\bib{Atangana}{misc}{
      author={Atangana, A.},
      author={Baleanu, D.},
       title={Caputo-fabrizio derivative applied to groundwater flow within
  confined aquifer},
        date={2017},
        note={Journal of Engineering Mechanics},
}

\bib{Atangana2}{misc}{
      author={Atangana, A.},
      author={Djida, J.D.},
       title={More generalized groundwater model with space‐time caputo
  fabrizio fractional differentiation},
        date={2017},
        note={Numerical methods for partial differential equations, 33},
}

\bib{GB}{misc}{
      author={Beylkin, G.},
      author={Monzon, L.},
       title={Approximation by exponential sums revisited},
        date={2010},
        note={Appl. Comput. Harmon. Anal., 28, 131–149},
}

\bib{Caputo1}{misc}{
      author={Caputo, M.},
       title={Linear model of dissipation whose q is almost frequency
  independent},
        date={1967},
        note={Geophys. J. R. Astr. Soc., VOL. 13.},
}

\bib{CaputoFabrizio}{misc}{
      author={Caputo, M.},
      author={Fabrizio, M.},
       title={On the notion of fractional derivative and applications to the
  hysteresis phenomena},
        date={2017},
        note={Meccanica 52, 3043–3052},
}

\bib{CCMH}{misc}{
      author={Carcione, J.},
      author={Cavallini, F.},
      author={Mainardi, F.},
      author={Hanyga, A.},
       title={Time-domain modeling of constant-q seismic waves using fractional
  derivatives},
        date={2002},
        note={Pure appl. geophys. 159, 1719–1736},
}

\bib{Carcione}{misc}{
      author={Carcione, J.M.},
       title={Theory and modeling of constant-q p- and s-waves using fractional
  time derivatives},
        date={2009},
        note={GEOPHYSICS, VOL. 74, NO. 1},
}

\bib{Carcione2001}{book}{
      author={Carcione, Jos\`{e}~M.},
       title={{Wave Fields in Real Media : Wave Propagation in Anisotropic,
  Anelastic, and Porous Media}},
        type={Book},
      series={Handbook of geophysical exploration. Seismic exploration,
  0950-1401 ; v. 31},
   publisher={Pergamon},
     address={Amsterdam ; New York},
        date={2001},
        ISBN={0080439292},
}

\bib{GS2}{misc}{
      author={Gao, G.H.},
      author={Sun, Z.Z.},
       title={A compact finite difference scheme for the fractional
  sub-diffusion equations},
        date={2011},
        note={J. Comput. Phys., 230, 586-595.},
}

\bib{GS}{misc}{
      author={Gao, G.H.},
      author={Sun, Z.Z.},
       title={The finite difference approximation for a class of fractional
  sub-diffusion equations on a space unbounded domain},
        date={2013},
        note={J. Comput. Phys., 236, 443-460.},
}

\bib{GSZ}{misc}{
      author={Gao, G.H.},
      author={Sun, Z.Z.},
      author={Zhang, Y.N.},
       title={A finite difference scheme for fractional sub-diffusion equations
  on an unbounded domain using artificial boundary conditions},
        date={2012},
        note={J. Comput. Phys., 231, 2865–2879.},
}

\bib{GHJCA}{misc}{
      author={Gu, X.M.},
      author={Huang, T.Z.},
      author={Ji, C.C.},
      author={Carpentieri, B.},
      author={Alikhanov, A.A.},
       title={Fast iterative method with a second order implicit difference
  scheme for time-space fractional convection-diffusion equations},
        date={2017},
        note={J. Sci. Comput, VOL. 72},
}

\bib{JiLiao}{misc}{
      author={Ji, B.Q.},
      author={Liao, H.L.},
      author={Gong, Y.Z.},
      author={Zhang, L.M.},
       title={Adaptive second-order crank-nicolson time-stepping schemes for
  time fractional molecular beam epitaxial growth models},
        date={2019},
        note={SIAM J. Sci. Comput., 42(3), B738–B760.},
}

\bib{JiangGreengardWang}{misc}{
      author={Jiang, S.},
      author={Greengard, L.},
      author={Wang, S.},
       title={Efficient sum-of-exponentials approximations for the heat kernel
  and their applications},
        date={2015},
        note={Adv. Comput. Math., 41, 529-551},
}

\bib{JiangZhang}{article}{
      author={Jiang, Shidong},
      author={Zhang, Jiwei},
      author={Zhang, Qian},
      author={Zhang, Zhimin},
       title={Fast evaluation of the caputo fractional derivative and its
  applications to fractional diffusion equations},
        date={2017},
     journal={Communications in Computational Physics},
      volume={21},
      number={3},
       pages={650\ndash 678},
}

\bib{KST}{misc}{
      author={Kilbas, A.A.},
      author={Srivastava, H.M.},
      author={Trujillo, J.J.},
       title={Theory and applications of fractional differential equations},
        date={2006},
        note={Elsevier, Amsterdam, VOL. 13.},
}

\bib{Kjartansson}{misc}{
      author={Kjartansson, E.},
       title={Constant q‐wave propagation and attenuation},
        date={1979},
        note={J. Geophys, VOL. 84, No. 89},
}

\bib{LiuMaZhou}{misc}{
      author={Liu, J.G.},
      author={Ma, Z.},
      author={Zhou, Z.Z.},
       title={Explicit and implicit tvd schemes for conservation laws with
  caputo derivatives},
        date={2017},
        note={J. Sci. Comput, VOL. 72},
}

\bib{LoskotBeaulieu}{misc}{
      author={Loskot, P.},
      author={Beaulieu, N.C.},
       title={Prony and polynomial approximations for evaluation of the average
  probability of error over slow-fading channels},
        date={2009},
        note={IEEE COMMUNICATIONS LETTERS, VOL. 58, NO. 3},
}

\bib{MKM}{misc}{
      author={Machado, J.T.},
      author={Kiryakova, V.},
      author={Mainardi, F.},
       title={Recent history of fractional calculus},
        date={2011},
        note={Communications in Nonlinear Science and Numerical Simulation},
}

\bib{OLBC}{misc}{
      author={Olver, F.W.J.},
      author={Lozier, D.W.},
      author={Boisvert, R.F.},
      author={Clark, C.W.},
       title={Nist handbook of mathematical functions},
        date={2010},
        note={Cambridge University Press, New York, NY},
}

\bib{Podlubny}{misc}{
      author={Podlubny, I.},
       title={Fractional differential equations},
        date={1999},
        note={vol. 198 of Mathematics in Science and Engineering, Academic
  Press.},
}

\bib{SAAMD}{article}{
      author={Qureshi, S.},
      author={Yusuf, A.},
      author={Shaikh, A.A.},
      author={Inc, M.},
      author={Baleanu, D.},
       title={Fractional modeling of blood ethanol concentration system with
  real data application},
        date={2019},
     journal={Chaos: An Interdisciplinary Journal of Nonlinear Science},
      volume={29},
      number={1},
       pages={013143},
}

\bib{SKM}{misc}{
      author={Samko, S.G.},
      author={Kilbas, A.A.},
      author={Marichev, O.I.},
       title={Fractional integrals and derivatives: Theory and applications},
        date={1993},
        note={Gordon and Breach Science Publishers, Switzerland},
}

\bib{shen2018fast}{article}{
      author={Shen, Jin-ye},
      author={Sun, Zhi-zhong},
      author={Du, Rui},
       title={Fast finite difference schemes for time-fractional diffusion
  equations with a weak singularity at initial time},
        date={2018},
     journal={East Asian J. Appl. Math},
      volume={8},
      number={4},
       pages={834\ndash 858},
}

\bib{SousaLi}{misc}{
      author={Sousa, Ercilia.},
      author={Li, Can},
       title={A weighted finite difference method for the fractional diffusion
  equation based on the riemann-liouville derivative},
        date={2015},
        note={Applied Numerical Mathematics, VOL. 90, 22-37},
}

\bib{SunChenWei}{misc}{
      author={Sun, H.G.},
      author={Chen, W.},
      author={Wei, H.},
      author={Chen, Y.Q.},
       title={A comparative study of constant-order and variable-order
  fractional models in characterizing memory property of systems},
        date={2011},
        note={Eur. Phys. J. Spec. Top. 193, 185},
}

\bib{sun2006fully}{article}{
      author={Sun, Zhi-zhong},
      author={Wu, Xiaonan},
       title={A fully discrete difference scheme for a diffusion-wave system},
        date={2006},
     journal={Applied Numerical Mathematics},
      volume={56},
      number={2},
       pages={193\ndash 209},
}

\bib{Wang}{misc}{
      author={Wang, C.L.},
      author={Wang, Z.Q.},
      author={Wang, L.L.},
       title={A spectral collocation method for nonlinear fractional boundary
  value problems with a caputo derivative},
        date={2018},
        note={J. Sci. Comput, VOL. 76.},
}

\bib{WangTreena}{misc}{
      author={Wang, H.},
      author={Treena, S.B.},
       title={A fast finite difference method for two-dimensional
  space-fractional diffusion equations},
        date={2012},
        note={SIAM J. Sci. Comput., 34, 2444–2458.},
}

\bib{ASMADA}{article}{
      author={Yusuf, A.},
      author={Qureshi, S.},
      author={Inc, M.},
      author={Aliyu, A.I.},
      author={Baleanu, D.},
      author={Shaikh, A.A.},
       title={Two-strain epidemic model involving fractional derivative with
  mittag-leffler kernel},
        date={2018},
     journal={Chaos: An Interdisciplinary Journal of Nonlinear Science},
      volume={28},
      number={12},
       pages={123121},
}

\bib{ZengTB}{misc}{
      author={Zeng, F.H.},
      author={Turner, I.},
      author={Burrage, K.},
       title={A stable fast time-stepping method for fractional integral and
  derivative operators},
        date={2018},
        note={J. Sci. Comput, VOL. 77},
}

\bib{ZZFL}{misc}{
      author={Zhan, Q.},
      author={Zhuang, M.},
      author={Fang, Y.},
      author={Liu, Q.~H.},
       title={{Discontinuous Galerkin modeling of 3D arbitrary anisotropic Q}},
        date={2019},
        note={GEOPHYSICS, VOL. 84, NO. 6},
}

\bib{ZhanZhuangLiu}{misc}{
      author={Zhan, Q.},
      author={Zhuang, M.},
      author={Liu, Q.~H.},
       title={{Adaptive Discontinuous Galerkin Modeling of Intrinsic
  Attenuation Anisotropy for Fluid-Saturated Porous Media}},
        note={IEEE Trans. GEOSCI REMOTE},
}

\bib{ZhanZhuangSun2017}{article}{
      author={Zhan, Qiwei},
      author={Zhuang, Mingwei},
      author={Sun, Qingtao},
      author={Ren, Qiang},
      author={Ren, Yi},
      author={Mao, Yiqian},
      author={Liu, Qing~Huo},
       title={{Efficient ordinary differential equation-based discontinuous
  Galerkin method for viscoelastic wave modeling}},
        date={2017},
        ISSN={0196-2892},
     journal={IEEE Transactions on Geoscience and Remote Sensing},
      volume={55},
      number={10},
       pages={5577\ndash 5584},
}

\bib{ZHAO2019531}{article}{
      author={Zhao, D.Z.},
      author={Luo, M.K.},
       title={Representations of acting processes and memory effects: General
  fractional derivative and its application to theory of heat conduction with
  finite wave speeds},
        date={2019},
        ISSN={0096-3003},
     journal={Applied Mathematics and Computation},
      volume={346},
       pages={531\ndash 544},
}

\end{biblist}
\end{bibdiv}

\end{document}